\documentclass{amsart}

\usepackage[dvipdfmx]{graphicx}

\usepackage{enumerate}
\usepackage{amsmath,amssymb}
\usepackage{ascmac}
\usepackage{here} 
\usepackage{color}

\newtheorem{mtheorem}{Theorem}

\newtheorem{theorem}{Theorem}[section]
\newtheorem{lemma}[theorem]{Lemma}
\newtheorem{proposition}[theorem]{Proposition}

\theoremstyle{definition}
\newtheorem{definition}[theorem]{Definition}

\newtheorem{remark}[theorem]{Remark}
\newtheorem{question}{Question}

\numberwithin{equation}{section}
\numberwithin{figure}{section}
\numberwithin{table}{section}

\let\oldmarginpar\marginpar
\renewcommand{\marginpar}[1]{%
\oldmarginpar{\color{red}\footnotesize%
\bfseries\sffamily #1}%
}

\begin{document}

\title[Moduli of diffeomorphisms with cubic tangencies]{
Moduli of surface diffeomorphisms with cubic tangencies}

\author{Shinobu Hashimoto} 
\address{Department of Mathematics and Information Sciences,
Tokyo Metropolitan University,
Minami-Ohsawa 1-1, Hachioji, Tokyo 192-0397, JAPAN}
\email{hashimoto-shinobu@ed.tmu.ac.jp }

\subjclass[2010]{Primary: 37C15, 37C70, 37E30}
\keywords{diffeomorphism, moduli, cubic tangency}

\date{\today}

\begin{abstract}
In this paper, we study conjugacy invariants for 2-dimensional diffeomorphisms with 
homoclinic cubic tangencies (two-sided tangencies of the lowest order) under certain open conditions.
Ordinary arguments used in past studies of conjugacy invariants associated with one-sided tangencies 
do not work in the two-sided case.
We present a new method which is applicable to the two-sided case.
\end{abstract}

\maketitle

\section{Introduction}\label{Intro}

Let $M$ be a closed surface and $\mathrm{Diff}^r(M)$ the space of 
$C^r$-diffeomorphisms with $C^r$-topology.
Suppose that $f_i$ $(i=0,1)$ are elements of $\mathrm{Diff}^2(M)$ with two saddle fixed points $p_i$, $q_i$.
We consider the case where $W^u(p_i)$ and $W^s(q_i)$ have a quadratic heteroclinic tangency $r_i$ 
and there exists a homeomorphism $h:M\to M$ with $f_1=h\circ f_0\circ h^{-1}$,  
$h(p_0)=p_1$, $h(q_0)=q_1$ and $h(r_0)=r_1$.
Then, under some moderate conditions, Palis \cite{pa} proved that
$\dfrac{\log |\lambda_0|}{\log |\mu_0|}= \dfrac{\log |\lambda_1|}{\log |\mu_1|}$, 
where $\lambda_i$ is the contracting eigenvalue of $Df(p_i)$ and 
$\mu_i$ is the expanding eigenvalue of $Df(q_i)$.
Such topological conjugacy invariants are often called \emph{moduli}.
Following his result, de Melo \cite{dm} studied moduli of stability of two-dimensional diffeomorphisms 
$f$, that is, 
a minimal set of moduli which parametrize topological conjugacy classes of a neighborhood of 
$f$ in $\mathrm{Diff}^r(M)$.
He detected moduli of stability for some classes of two-dimensional diffeomorphisms.
In \cite{dm}, he also showed that the restriction of the conjugacy homeomorphism $h$ on  
$W^u(p_0)\setminus \{p_0\}$ is a local diffeomorphism if $\dfrac{\log |\lambda_0|}{\log |\mu_0|}$ is irrational.
Subsequently, Posthumus \cite{po} proved that the homoclinic version of the results due to Palis and de Melo.
In fact, he proved that, if $f_i$ $(i=0,1)$ has a saddle fixed point $p_i$ with a homoclinic quadratic tangency $r_i$, 
then $\dfrac{\log |\lambda_0|}{\log |\mu_0|}= \dfrac{\log |\lambda_1|}{\log |\mu_1|}$ holds, 
where $\lambda_i,\mu_i$ are the contracting and expanding eigenvalues of $Df(p_i)$.
Moreover, if $\dfrac{\log |\lambda_0|}{\log |\mu_0|}$ is irrational, then the eigenvalues are also moduli, 
that is, $\lambda_0=\lambda_1$ and $\mu_0=\mu_1$.
Various results related to moduli concerning eigenvalues are obtained by some authors; see for example \cite{dmp, dmvs,pot,gpvs}.
However, in all of these results, the assumption that the tangency is quadratic or one-sided is crucial.
In fact, some of their arguments do not work in the case that $q$ is a two-sided tangency, 
see Remark \ref{r_one_sided} for the reason.

Here we consider surface diffeomorphisms of $C^3$-class with cubic homoclinic tangencies, which are 
two-sided tangencies of lowest order.
Bifurcations of such a tangency is one of typical subjects in 2-dimensional dynamics, 
see for example \cite{kky, ks1}.
Moreover, Kiriki and Soma \cite{ks2} presented infinitely many (original) H\'enon maps 
which have homoclinic cubic tangencies.
In this paper, we will prove some results for surface diffeomorphisms with cubic homoclinic tangencies 
%under some conditions 
corresponding to those 
for diffeomorphisms with one-sided homoclinic tangencies in \cite{pa,dm,pot} .
As far as the author knows, this paper is the first attempt to study moduli 
with respect to two-sided homoclinic tangencies. 
Our argument works if the expanding eigenvalue of $Df(p)$ is  
sufficiently small, where $p$ is a saddle fixed point with a homoclinic cubic tangency.

%From another point of view, our result presents the complementary part of results by 
%Kiriki and Soma \cite{ks1, ks2}.
%In those papers, they studied families of surface diffeomorphisms which unfold 
%cubic tangencies.
%On the other hand, our result studies the possibility of 
%perturbation of surface diffeomorphisms without breaking cubic tangencies.

\begin{mtheorem}\label{thm_A}
Suppose that $M$ is a closed surface with Riemannian metric.
Let $f_i$ $(i=0,1)$ be elements of $\mathrm{Diff}^3(M)$ each of which has 
a saddle fixed point $p_i$ and a homoclinic cubic tangency $q_i$ associated with $p_i$ 
and satisfies the following conditions.
\begin{enumerate}
\renewcommand{\theenumi}{A\arabic{enumi}}
\renewcommand{\labelenumi}{\rm(\theenumi)}
\item\label{A1}
For $i=0,1$, there exists a neighborhood $U(p_i)$ of $p_i$ in $M$ such that $f|_{U(p_i)}$ is linear. 
\item\label{A2}  
$f_0$ is topologically conjugate to $f_1$ by a homeomorphism $h:M\to M$ with $h(p_0)=p_1$ 
and $h(q_0)=q_1$.
\item\label{A3}  
Each $f_i$ $(i=0,1)$ satisfy the small expanding condition and one of the 
adaptable conditions with respect to $(p_i,q_i)$ in Section \ref{S_adaptable}.
\end{enumerate}
Then \eqref{M1} and $\eqref{M2}$ hold, 
where $\lambda_i$, $\mu_i$ are the eigenvalues of $Df_0(p_i)$ 
with $0<|\lambda_i|<1<|\mu_i|$.
\begin{enumerate}
\renewcommand{\theenumi}{M\arabic{enumi}}
\renewcommand{\labelenumi}{\rm(\theenumi)}
\item\label{M1}
$\dfrac{\log |\lambda_0|}{\log |\mu_0|}=\dfrac{\log |\lambda_1|}{\log |\mu_1|}$.
\item\label{M2}
Moreover, if $\dfrac{\log |\lambda_0|}{\log |\mu_0|}$ is irrational, then $\mu_0=\mu_1$ and $\lambda_0=\lambda_1$.
\end{enumerate}
\end{mtheorem}

Here we say that $f_0$ satisfies the \emph{small expanding condition} at $p_0$ if $|\mu_0|=
1+\varepsilon$ with $0<\varepsilon<\varepsilon_0$ for the constant $\varepsilon_0$ given in Lemma \ref{l_slow3'}.
Note that this condition depends on local expressions of $f_0$ such as \eqref{eqn_linear} near $p_0$ and \eqref{eqn_varphi} near $f_0^{m_0}(q_0)$.
In Section \ref{S_pre}, we present a codimension two submanifold $\mathcal{C}$ of 
$\mathrm{Diff}^3(M)$ such that any 
 element of $\mathcal{C}$ sufficiently close to $f_0$  
also satisfies \eqref{A3}.
In the case that $f$ is of class $C^\infty$,  we know from Sternberg \cite{st} and Takens \cite{ta} 
that \eqref{A1} 
is an open dense condition in $\mathrm{Diff}^\infty(M)$. 

Though we only consider the case of cubic tangencies,
we believe that our method still works in the case of two-sided tangencies of higher order.
So we propose the following question.

\begin{question}
Is it possible to generalize our theorems to the case that diffeomorphisms have 
two-sided homoclinic tangencies of higher order\,?
\end{question}

We will finish the introduction by outlining the proof of the main theorem.

Let $f_0$ be a diffeomorphism satisfying the conditions of Theorem \ref{thm_A}.
We may assume that $q_0$ and $r_0=\varphi(q_0)$ are contained in $W_{\mathrm{loc}}^u(p_0)$ and 
$W_{\mathrm{loc}}^s(p_0)$ respectively, where $\varphi=f_0^{m_0}$ for some positive integer $m_0$.
For the proof of Theorem \ref{thm_A}, we need to find out a useful connection between the 
eigenvalues $\mu_i$ and $\lambda_i$ for $i=0,1$.
By applying Inclination Lemma, we have a sequence $\{\alpha_n^u\}$ of arcs in $W^u(p_0)$ 
which meet $W_{\mathrm{loc}}^s(p_0)$ transversely at single points $z_0\lambda_0^n$ and 
$C^3$-converge to a sub-arc of $W_{\mathrm{loc}}^u(p_0)$, see Figure \ref{fig_alpha_n}.
Then $\varphi(\alpha_n^u)$ contains an S-shaped arc $\gamma_{0,n}'$ framed by the rectangle $S_n$ as illustrated in 
Figure \ref{fig_outline}.

%%%%%%%%%%%%%%%%%%%%%%%%%%%%%%%%%%%%%%%%%%%%%%%%%%%%%%%%%%%%
\begin{figure}[hbt]
\centering
\scalebox{0.3}{\includegraphics[clip]{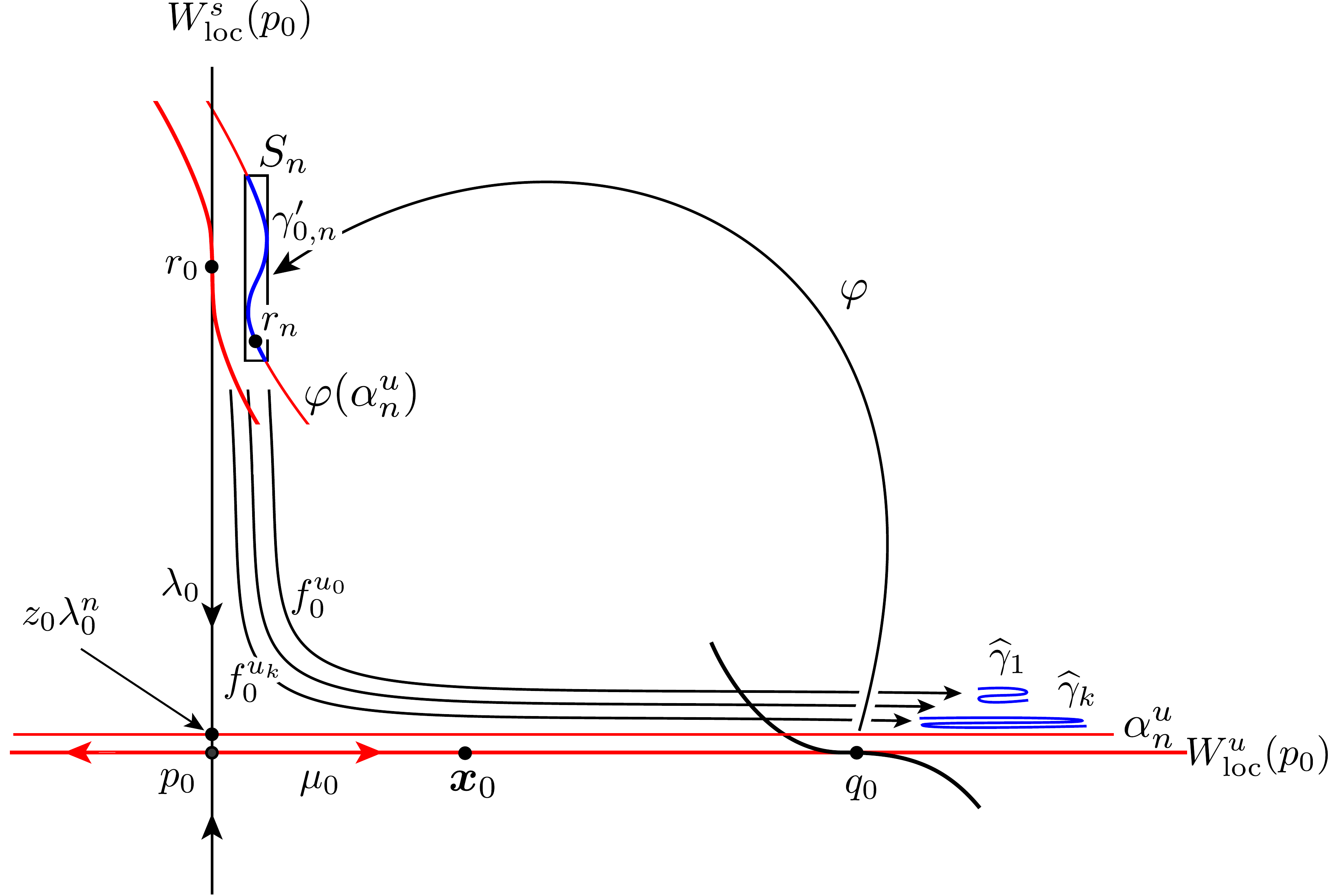}}
\caption{}
\label{fig_outline}
\end{figure}
%%%%%%%%%%%%%%%%%%%%%%%%%%%%%%%%%%%%%%%

We note that such arcs $\gamma_{0,n}'$ are subtle and vanish eventually as $n\to \infty$, 
see Figures \ref{fig_Sn} and \ref{fig_rn}.
Since $h$ is not supposed to be smooth, one can not expect that $h$ sends $\gamma_n'$ to an S-shaped curve in $W^u(p_1)$.
However Intersection Lemma (Lemma \ref{l_intersection}) shows that it actually holds, 
which is a key lemma in our argument.
For the proof, we send $\gamma_{0,n}'$ to a curve $\widehat\gamma_1$ in a small neighborhood of $q_0$ by $f_0^{u_0}$ for some $u_0\in \mathbb{N}$ and pull it back near $r_0$ by $\varphi$.
Repeating this process many times, one can amplify $\widehat\gamma_1$ and finally have a compressed S-shaped curve  $\widehat\gamma_k$ near $q$ the diameter of which is substantial so that it can be distinguished by $h$.
From this fact, we know that $h(\widehat\gamma_k)$ intersects a compressed S-shaped curve $\widehat\gamma_k^*$ in $W^u(p_1)$.
It follows that there exists a sequence $\{r_n\}$ with $r_n\in \gamma_{0,n}'$ as illustrated in 
Figure \ref{fig_outline} such that 
$\bar r_n=h(r_n)$ is contained in the corresponding S-shaped curve $\bar\gamma_{0,n}'$ in $W^u(p_1)$.
We note that the images of $r_n$, $\bar r_n$ by the orthogonal projections to the first coordinates 
are represented as $az_0\lambda_0^n+o(\lambda_0^n)$, $\bar a\bar z_0\lambda_1^n+o(\lambda_1^n)$ respectively 
for some non-zero constants $a$, $\bar a$.
One can take subsequences $\{n(k)\}$, $\{m(k)\}$ of $\mathbb{N}$ such that $f_0^{m(k)}(r_{n(k)})$ 
converges to a point $x_0\in W_{\mathrm{loc}}^u(p_0)$.
Then $f_1^{m(k)}(\bar r_{n(k)})$ 
also converges to $h(x_0)\in W_{\mathrm{loc}}^u(p_1)$.
By using this fact, we will show that   
$\lim_{k\to \infty}\dfrac{m(k)}{n(k)}=-\dfrac{\log \lambda_0}{\log \mu_0}$ and 
$\lim_{k\to \infty}\dfrac{m(k)}{n(k)}=-\dfrac{\log \lambda_1}{\log \mu_1}$.
This proves the assertion \eqref{M1}.
The assertion \eqref{M2} is proved by \eqref{M1} together with standard arguments in \cite{dm,po}.

\section{Preliminaries}\label{S_pre}

Let $\{a_n\}$, $\{b_n\}$ be sequences with non-zero entries.
Then $a_n\approx b_n$ means that $\dfrac{a_n}{b_n}\to 1$ as $n\to \infty$, 
and $a_n\sim b_n$ means that there exist constants $C$ and $C'$ independent of $n$ with $0<C'<1<C$ 
and satisfying $C' \leq \dfrac{a_n}{b_n}\leq C$ for any $n$.
Suppose next that $\{a_n\}$, $\{b_n\}$ are sequences with non-negative entries.
If there exists a constant $C'>0$ independent of $n$ and satisfying 
$a_n\leq C' b_n$ for any $n$, then we denote the property by $a_n\precsim b_n$.

Throughout the remainder of this paper, we suppose that $M$ is a closed connected surface and $f:M\to M$ 
is a $C^3$-diffeomorphism with a saddle fixed point $p$.
Let $\mu$, $\lambda$ be the eigenvalues of $Df(p)$ with 
\begin{equation}\label{eqn_mulambda}
0<|\lambda|<1<|\mu|.
\end{equation}
Suppose moreover that $f$ is $C^3$-linearizable in a neighborhood $U(p)$ of 
$p$ in $M$.
Then there exists a $C^3$-coordinate $(x,y)$ on $U(p)$ satisfying the following condition: 
\begin{equation}\label{eqn_linear}
f(x,y)=(\mu x,\lambda y)
\end{equation}
for any $(x,y)\in U(p)$.
In particular, this implies that $p=(0,0)$, $W_{\mathrm{loc}}^u(p):=\{(x,y)\in U(p);\,y=0\}\subset W^u(p)$ and  
$W_{\mathrm{loc}}^s(p):=\{(x,y)\in U(p);\,x=0\}\subset W^s(p)$.

A non-transverse intersection point $q$ of $W^u(p)$ and $W^s(p)$ is called a \emph{homoclinic tangency} associated 
with $p$.
We fix a Riemannian metric on $M$ and denote the distance on $M$ induced from the metric by $d$. 
The tangency is of \emph{order} $n$ if the limit
$$\lim_{\begin{subarray}{l} w\in W_{\mathrm{loc}}^s(p),\\[2pt] w\to q\end{subarray}}\frac{d(w,W^u(p))}{[d(w,q)]^n}
$$
exists and has non-zero value.
See \cite[Section 2]{po} for the definition.
If $n=2$ (resp.\ $n=3$), then the tangency $q$ is called  \emph{quadratic} (resp.\ \emph{cubic}).
Let $\mathcal{C}$ be the subspace of $\mathrm{Diff}^3(M)$ consisting of elements $f\in \mathrm{Diff}^3(M)$ satisfying the following conditions \eqref{C1}--\eqref{C3}.
\begin{enumerate}
\renewcommand{\theenumi}{C\arabic{enumi}}
\item\label{C1}
$f$ has a saddle periodic point $p$.
\item\label{C2}
There exists a homoclinic cubic tangency $q$ associated with $p$.
\item\label{C3}
$f$ satisfies the adaptable conditions in the sense of Section \ref{S_adaptable} with respect to $p,q$.
\end{enumerate}
Note that $\mathcal{C}$ is a codimension two submanifold of $\mathrm{Diff}^3(M)$.

Let $q$ be a cubic tangency of $W^u(p)$ and $W^s(p)$.
We assume that $q$ is contained in $W_{\mathrm{loc}}^u(p)  \subset U(p)$ if necessary replacing $q$ by $f^{-n}(q)$ with sufficiently large $n\in\mathbb{N}$.
For the point $q$, there exists $m_0\in \mathbb{N}$ such that $r:=f^{m_0}(q)\in W_{\mathrm{loc}}^s(p) \subset U(p)$.
Then one can rearrange the linearizing coordinate on $U(p)$ so that $q=(1,0)$, $r=(0,1)$.
Moreover, we may suppose that
$$U(p)=[-2,2]\times [-2,2],\ W_{\mathrm{loc}}^u(p)=[-2,2]\times \{0\},\ W_{\mathrm{loc}}^s(p)=\{0\}\times [-2,2].$$
Let $U(q)$, $U(r)$ be sufficiently small neighborhoods of $q$, $r$ in $U(p)$ respectively. 
Then the component $L^s(q)$ of $W^s(p)\cap U(q)$ containing $q$ is represented as 
$$L^s(q)= \{(x+1, y)\in U(q);\, y = v(x)\},$$
where $v$ is a $C^3$-function satisfying
\begin{equation}\label{eqn_vvv}
v(0) = v'(0) =v''(0) = 0\quad\text{and}\quad v'''(0)\neq 0.
\end{equation}
Similarly, the component $L^u(r)$ of $W^u(p)\cap U(r)$ containing $r$ is represented as 
$$L^u(r)= \{(x, y+1)\in U(r);\, x = w(y)\},$$
where $w$ is a $C^3$-function satisfying
\begin{equation}\label{eqn_www}
w(0) = w'(0) =w''(0) = 0\quad\text{and}\quad w'''(0)\neq 0,
\end{equation}
see Figure \ref{fig_vw0}

%%%%%%%%%%%%%%%%%%%%%%%%%%%%%%%%%%%%%%%
\begin{figure}[hbt]
\centering
\scalebox{0.3}{\includegraphics[clip]{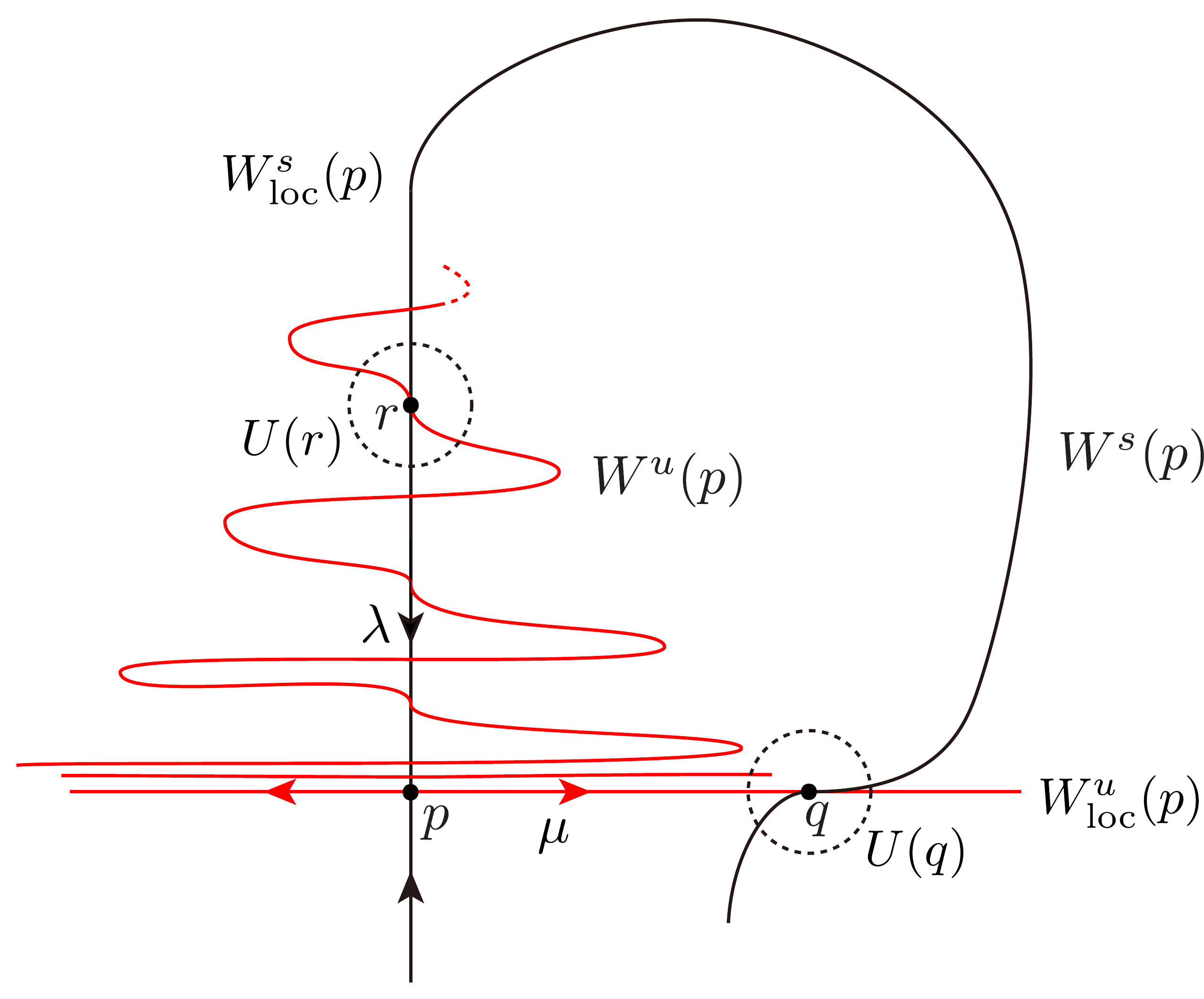}}
\caption{}
\label{fig_vw0}
\end{figure}
%%%%%%%%%%%%%%%%%%%%%%%%%%%%%%%%%%%%%%%

Recall that $q=(1,0)$, $r=(0,1)$ are cubic tangencies between $W^s(p)$ and $W^u(p)$ and $f^{m_0}(q)=r$ for some 
$m_0\in\mathbb{N}$.
We set $f^{m_0}=\varphi$ for short.
By \eqref{eqn_vvv} and \eqref{eqn_www}, $\varphi$ is represented in $U(q)$ as follows for some constants $a,b,c,d,e$.
\begin{equation}\label{eqn_varphi}
\varphi(x+1,y)=(ay+bxy+cx^3+H_1(x+1,y), 1+dx+ey+H_2(x+1,y)).
\end{equation}
where $H_1$, $H_2$ are $C^3$-functions 
satisfying the following conditions.
\begin{equation}\label{eqn_H1H2}
\begin{split}
H_1(1,0)&=\partial_x H_1(1,0)=\partial_y H_1(1,0)=\partial_{xx} H_1(1,0)=\partial_{xy} H_1(1,0)\\
&=\partial_{xxx} H_1(1,0)= 0,\\
H_2(1,0)&=\partial_x H_2(1,0)=\partial_y H_2(1,0)=0.
\end{split}
\end{equation}
Since $\varphi$ is a diffeomorphism, 
$$a,d\neq 0.$$
The fact that $q$ is a cubic tangency implies
$$c\neq 0.$$
Here we put the following extra open condition.
\begin{equation}\label{EX1}
b\neq 0.
\end{equation}

By \eqref{eqn_varphi} and \eqref{eqn_H1H2}, the Jacobian matrix of
$\varphi$ at $(x+1,y)$ is given as follows. 
\begin{equation}\label{eqn_dphi}
\begin{split}
d\varphi_{(x+1,y)}&=
\begin{bmatrix} 
\,by+3cx^2+\partial_x H_1(x+1,y)  & a+bx+\partial_y H_1(x+1,y)\,\\ 
d+\partial_x H_2(x+1,y) & e+\partial_y H_2(x+1,y)
\end{bmatrix}\\
&=
\begin{bmatrix} 
\,by+3cx^2+o(x^2)+o(y)+O(xy)& a+bx+o(x)+O(y)\,\\ 
d+O(x)+O(y) & e+O(x)+O(y)
\end{bmatrix}.
\end{split}
\end{equation}
Here we only consider the case satisfying the following condition, 
which belongs to Case $\mathrm{II}_{++}$ in Section \ref{S_adaptable}.
\begin{equation}\label{eqn_adaptable_0}
0<\lambda<1, \mu>1,a>0,b<0,c>0,d<0.
\end{equation}
See Figure \ref{fig_case0} for the situation of $W_{\mathrm{loc}}^u(p)$ and $W_{\mathrm{loc}}^s(p)$ 
in the case of \eqref{eqn_adaptable_0}.
Note that \eqref{eqn_adaptable_0} implies the extra condition \eqref{EX1}.
%%%%%%%%%%%%%%%%%%%%%%%%%%%%%%%%%%%%%%%
\begin{figure}[hbt]
\centering
\scalebox{0.3}{\includegraphics[clip]{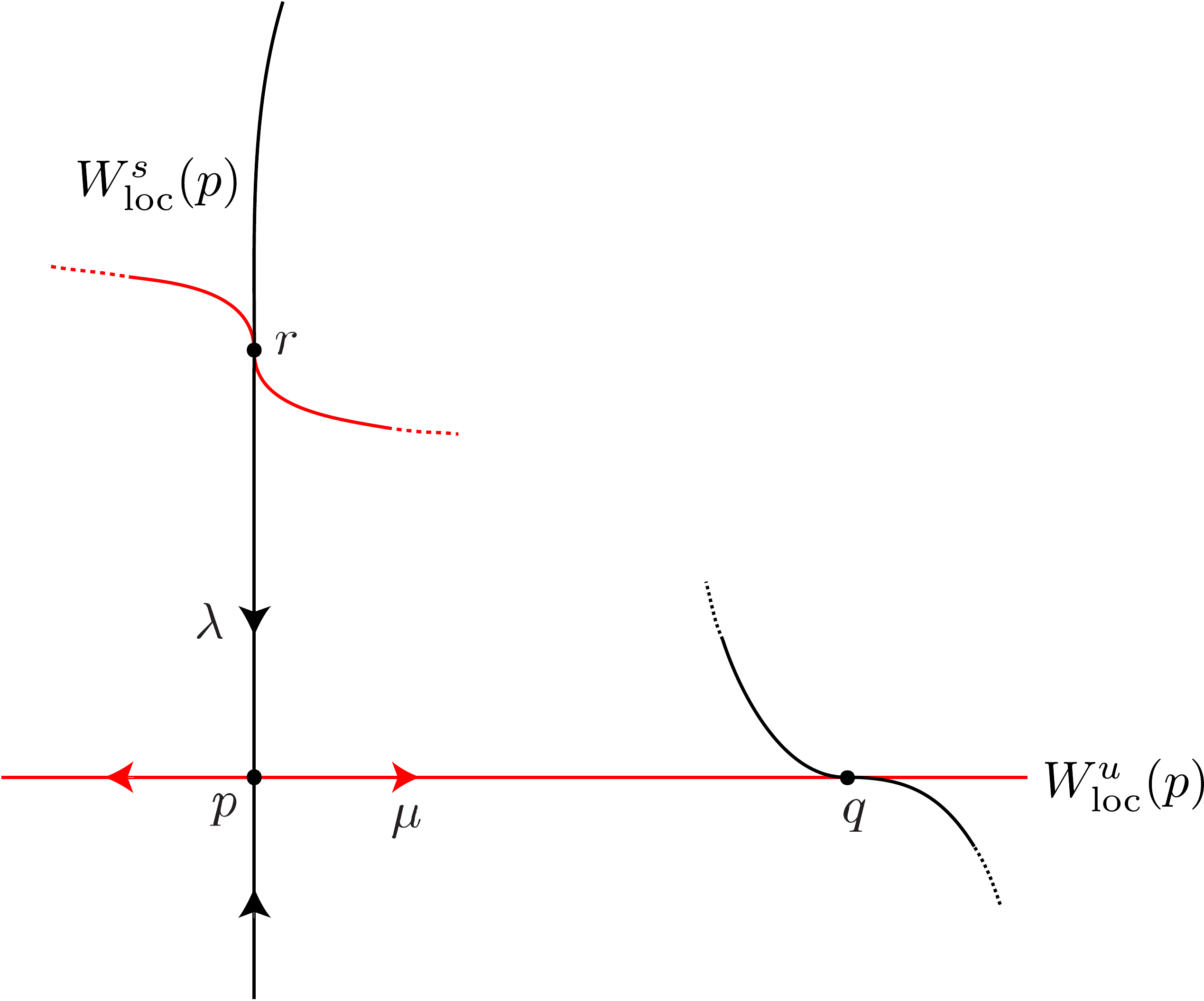}}
\caption{}
\label{fig_case0}
\end{figure}
%%%%%%%%%%%%%%%%%%%%%%%%%%%%%%%%%%%%%%%

One can set $\mu=1+\varepsilon$ for some $\varepsilon>0$.
We only consider the case that $\varepsilon$ is sufficiently small.

Consider the rectangle $R_\varepsilon=[1+\varepsilon,(1+\varepsilon)^3]\times [0,\varepsilon^3]$ in $U(q)$.
By \eqref{eqn_varphi}, 
\begin{equation}\label{eqn_vp(1+e)}
\begin{split}
\varphi(1+\varepsilon,0)&= (c\varepsilon^3+o(\varepsilon^3),1+d\varepsilon+o(\varepsilon)),\\
\varphi(1+\varepsilon,\varepsilon^3)&= ((a+c)\varepsilon^3+o(\varepsilon^3),1+d\varepsilon+o(\varepsilon)),\\
\varphi((1+\varepsilon)^3,0)&=(27c\varepsilon^3+o(\varepsilon^3), 1+3d\varepsilon+o(\varepsilon)),\\
\varphi((1+\varepsilon)^3,\varepsilon^3)&=((a+27c)\varepsilon^3+o(\varepsilon^3), 1+3d\varepsilon+o(\varepsilon)).
\end{split}
\end{equation}
Let ${\rm pr}_x : U(p) \rightarrow W^u_{\mathrm{loc}}(p)$ and ${\rm pr}_y : U(p) \rightarrow W^s_{\mathrm{loc}}(p)$ 
be the orthogonal projections with respect to the linearizing coordinate on $U(p)$.
Then there exist constants $\tau_0$, $\tau_1$ with $0<\tau_0<\tau_1$ independent of $\varepsilon$ and 
satisfying 
\begin{equation}\label{eqn_prxR}
\mathrm{pr}_x(\varphi(R_\varepsilon))\subset [\tau_0\varepsilon^3,\tau_1\varepsilon^3].
\end{equation}
Since $d<0$ by \eqref{eqn_adaptable_0}, it follows from \eqref{eqn_vp(1+e)} that
\begin{equation}\label{eqn_prxR}
\mathrm{pr}_y(\varphi(R_\varepsilon))\subset [1+3.5d\varepsilon,1+0.5d\varepsilon]\subset [1+4d\varepsilon,1].
\end{equation}

For any $\boldsymbol{x}\in R_\varepsilon$, let $u_0=u_0(\boldsymbol{x})$ 
be a uniquely determined positive integer such that 
$f^i(\varphi(\boldsymbol{x}))\in U(p)$ for $i=1,\dots,u_0$ and $\mathrm{pr}_x(f^{u_0}(\varphi(\boldsymbol{x})))
\subset ((1+\varepsilon)^2,(1+\varepsilon)^3]$.
Since $\mathrm{pr}_x(f^{u_0}(\varphi(\boldsymbol{x})))=\mu^{u_0}\mathrm{pr}_x(\varphi(\boldsymbol{x}))$, 
$$1<(1+\varepsilon)^2 < \mu^{u_0}\mathrm{pr}_x(\varphi(\boldsymbol{x})) < \tau_1\mu^{u_0}\varepsilon^3.$$
Since $\mathrm{pr}_y(\varphi(\boldsymbol{x})) < 1$ by \eqref{eqn_prxR}, 
it follows that 
$$\mathrm{pr}_y(f^{u_0}(\varphi(\boldsymbol{x})))=\lambda^{u_0}\mathrm{pr}_y(\varphi(\boldsymbol{x})) < \lambda^{u_0}.$$
Consider the following conditions for $\varepsilon > 0$: 
\begin{equation}
\tau_1 < \varepsilon^{-1}\quad\text{and}\quad (1+\varepsilon)^{\frac32} = \mu^{\frac32} < \lambda^{-1}.
\label{eqn_slow1}
\end{equation}
If these conditions are satisfied, then the following inequalities 
\begin{equation}\label{eqn_muu0}
1<1+\varepsilon < \mu^{u_0}\mathrm{pr}_x(\varphi(\boldsymbol{x})) < \tau_1\mu^{u_0}\varepsilon^3 < \mu^{u_0}\varepsilon^2
 \end{equation}
hold.
This implies that
$$\mathrm{pr}_y(f^{u_0}(\varphi(\boldsymbol{x})))=\lambda^{u_0}\mathrm{pr}_y(\varphi(\boldsymbol{x})) < \lambda^{u_0} < \mu^{-\frac{3}{2}u_0} < \varepsilon^3.$$
Thus the positive integer $u_0(\boldsymbol{x})$ satisfies
\begin{equation}\label{eqn_Rve}
f^{u_0(\boldsymbol{x})}(\varphi(\boldsymbol{x})) \in R_\varepsilon
\end{equation}
for all $\boldsymbol{x} \in R_\varepsilon$.
See Figure \ref{fig_Re}.

%%%%%%%%%%%%%%%%%%%%%%%%%%%%%%%%%%%%%%%
\begin{figure}[hbt]
\centering
\scalebox{0.3}{\includegraphics[clip]{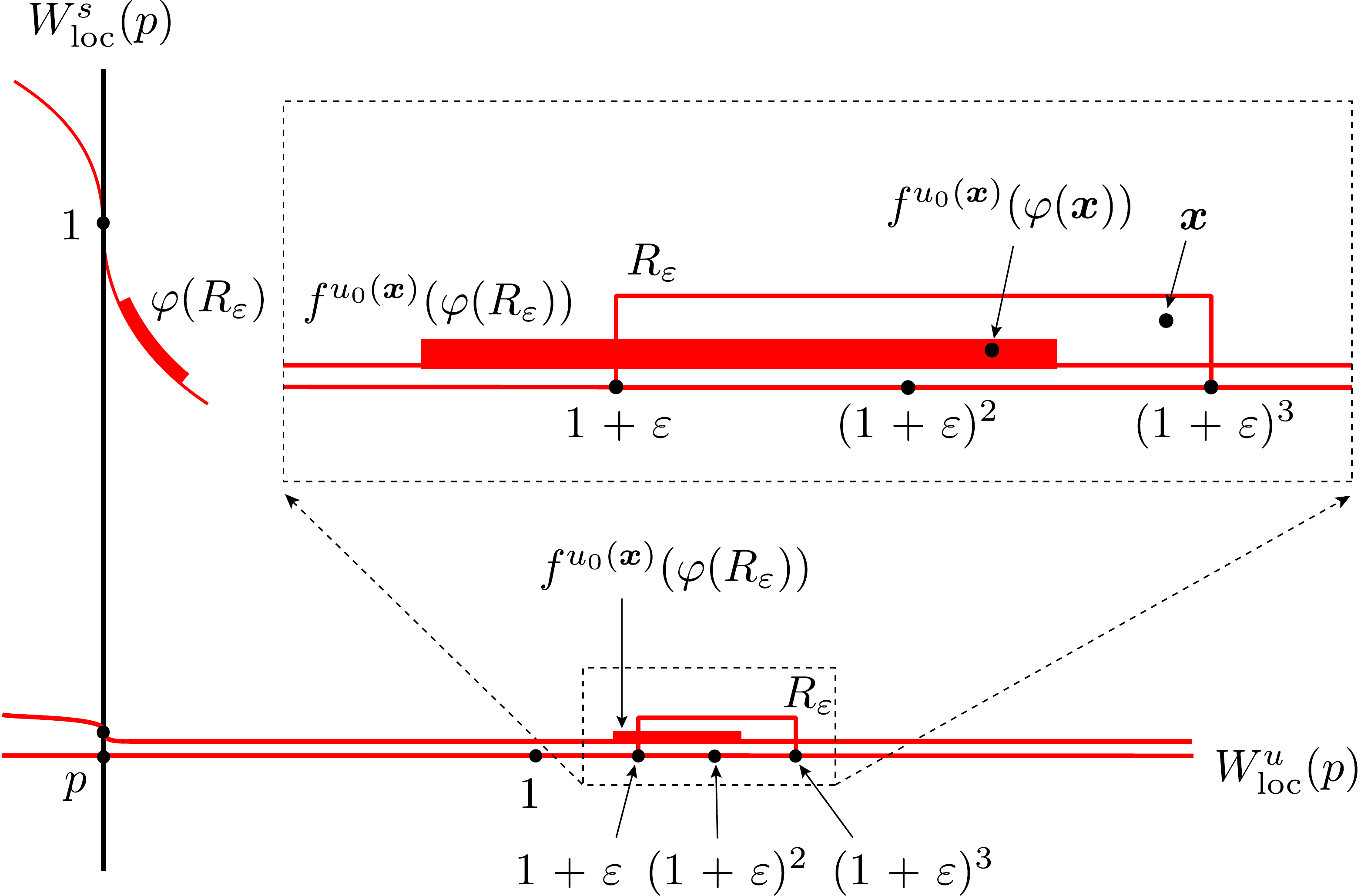}}
\caption{}
\label{fig_Re}
\end{figure}
%%%%%%%%%%%%%%%%%%%%%%%%%%%%%%%%%%%%%%%

\section{Sequence of Rectangles}\label{S_Sbox}

Let $f : M \rightarrow M$ be a $C^3$-diffeomorphism given in Section \ref{S_pre}.
In particular, $f$ satisfies the linearizing condition \eqref{eqn_linear} on $U(p)$.
It is not hard to show that $W^u(p)$ and $W^s(p)$ have a transverse intersection point other than $p$.
See, for example, Lemma 1.2 in \cite{ks1}.
Let $\delta^u$ be a segment in $W_{\mathrm{loc}}^u(p)$ with $\mathrm{Int}\delta^u \supset \{p,q\}$.
Then, by Inclination Lemma, 
there exists a sequence $\{\alpha_n^u\}_{n=0}^\infty$ of arcs in $W^u(p)$ 
$C^3$-converging to $\delta^u$ and satisfying the following conditions:
\begin{itemize}
\setlength{\leftskip}{-18pt}
\item
$\alpha_0^u$ meets $W_{\mathrm{loc}}^s(p)$ transversely in a single point $\boldsymbol{z}_0=(0,z_0)$.
\item
Each $\alpha_n^u$ contains $f^n(\boldsymbol{z}_0)=(0,z_0\lambda^n)$, and the intersection $\tilde\alpha_n^u=\alpha_n^u\cap U(q)$ is an arc meeting 
$L^s(q)$ transversely in a single point $c_n$ for any sufficiently large $n>0$.
\end{itemize}
See Figure \ref{fig_alpha_n}.
%%%%%%%%%%%%%%%%%%%%%%%%%%%%%%%%%%%%%%%
\begin{figure}[hbt]
\centering
\scalebox{0.3}{\includegraphics[clip]{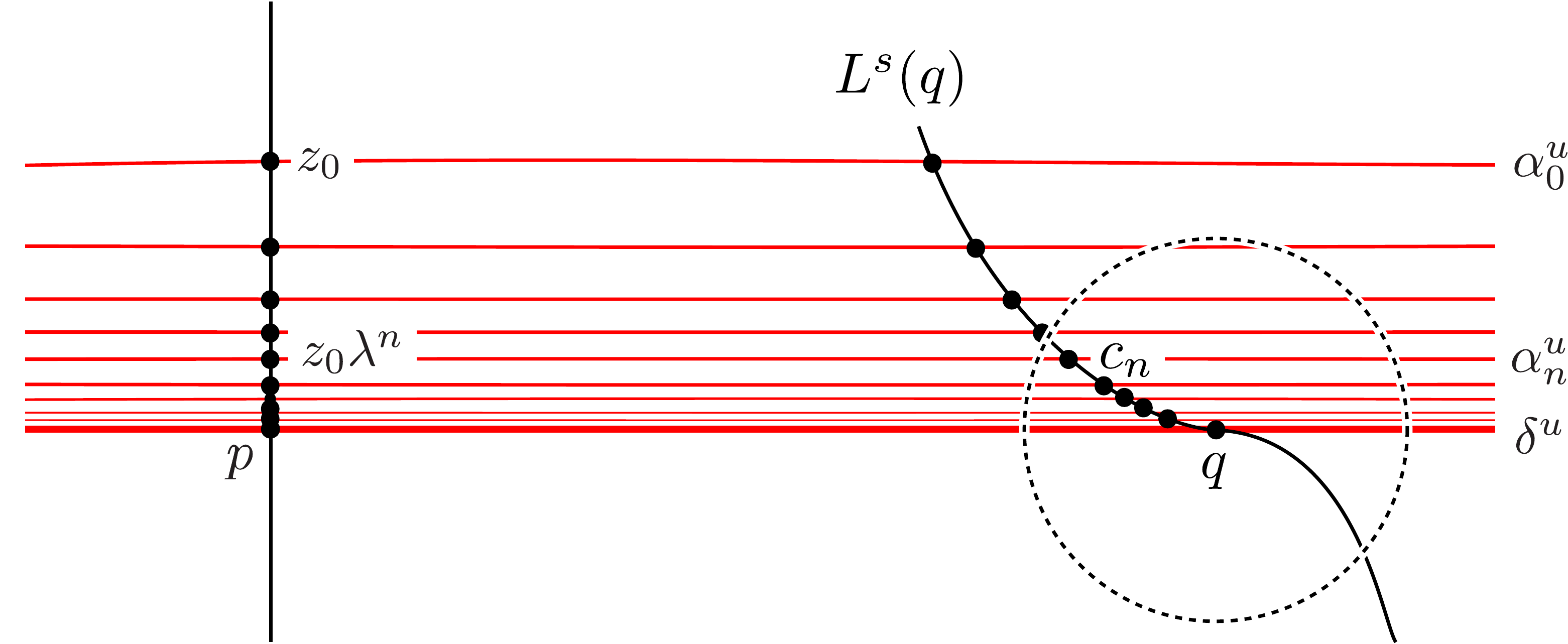}}
\caption{}
\label{fig_alpha_n}
\end{figure}
%%%%%%%%%%%%%%%%%%%%%%%%%%%%%%%%%%%%%%%
Note that $\alpha_0^u$ is represented by the graph of a $C^3$-function $y_0:\delta^u\to \mathbb{R}_+$, 
that is, $\alpha_0^u=\{(x,y_0(x))\,;\, x\in \delta^u\}$.  
Then each $\alpha_n^u$ is represented by the graph of the function $y_n:\delta^u\to \mathbb{R}_+$ 
with 
\begin{equation}\label{eqn_psi}
y_n(x)=\lambda^n y_0(\mu^{-n}x)\quad\text{for}\quad x\in \delta^u.
\end{equation}
We parametrise $\tilde\alpha^u_n$ in $[(1+\varepsilon)^{-3},(1+\varepsilon)^3]$ by $\alpha_n(t)=(t+1,\tilde y_n(t))$ with $
(1+\varepsilon)^{-3}-1\leq t\leq 
(1+\varepsilon)^3-1$, where $\tilde y_n(t)=y_n(t+1)$.
By \eqref{eqn_varphi} and \eqref{eqn_dphi}, 
\begin{align}
\label{eqn_va}\varphi ( \alpha_n(t))&=
(a\tilde y_n(t)+bt\tilde y_n(t)+ct^3+\mathrm{h.o.t.},1+dt+e\tilde y_n(t)+\mathrm{h.o.t.})\\
\label{eqn_va'}d\varphi_{\alpha_n(t)}(\alpha_n'(t))&=
(a\tilde y_n'(t)+b\tilde y_n(t)+bt\tilde y_n'(t)+3ct^2+\mathrm{h.o.t.},\\
&\hspace{150pt} d+e\tilde y_n'(t)+\mathrm{h.o.t.}),\nonumber
\end{align}
where the primes represent the derivative on $t$ and `h.o.t.' denotes the sum of the higher order terms on $t$.

By \eqref{eqn_psi},
$$|\tilde y_n'(t)|=|y_n'(t+1)| = \lambda^n\mu^{-n}|y_0^{\prime}(\mu^{-n}(t+1))|.$$
Suppose that $\sigma$ is the maximum of $|y_0^{\prime}(x)|$ on $\delta^u$.
Then 
$$|\tilde y_n'(t)|=|y_n'(t+1)| = \lambda^n\mu^{-n}|y_0^{\prime}(\mu^{-n}(t+1))| \le \lambda^n\mu^{-n}\sigma.$$
for any $n \in \mathbb{N}$.
This implies that 
\begin{equation}\label{eqn_y_n'}
|\tilde y_n'(t)| \precsim \lambda^n\mu^{-n}.
\end{equation}
Suppose that $d\varphi_{\alpha_n(t)} (\alpha_n'(t))$ is vertical at $t=t_n$.
Then $\lim_{n\to \infty}t_n=0$ and, by \eqref{eqn_va'},  
$$b\tilde y_n(t_n)+(a+bt_n)\tilde y_n'(t_n)\approx -3ct_n^2.$$
Since $\tilde y_n(t)\approx \lambda^nz_0$ and $|\tilde y_n'(t)| \precsim \lambda^n\mu^{-n}$, 
this condition is equivalent to
\begin{equation}\label{eqn_3ctn}
3ct_n^2\approx -b\tilde y_n(t_n)\approx -b\lambda^nz_0.
\end{equation}
It follows that, for all sufficiently large $n$, $d\varphi_{\alpha_n(t)} (\alpha_n'(t))$ is vertical at two points 
$t_{n,\pm}$ with 
\begin{equation}\label{eqn_tn+-}
t_{n,\pm}\approx \pm\sqrt{\dfrac{-bz_0}{3c}}\lambda^{\frac{n}2}.
\end{equation}
Let $\tilde t_{n,\pm}$ be the elements of  $[(1+\varepsilon)^{-3}-1, (1+\varepsilon)^3-1]$ 
with $\tilde t_{n,-} < t_{n,-}$, $ t_{n,+} < \tilde t_{n,+}$ 
such that $\varphi(\alpha_n(\tilde t_{n,\pm}))$ is the intersection point of $\varphi(\alpha_n(t))$ 
and the vertical line $L_{n,\pm}$ tangent to $\varphi(\alpha_n(t))$ at $\varphi(\alpha_n(t_{n,\mp}))$. 
Let $S_n$ be the smallest orthogonal rectangle in $U(r)$ 
containing the four points 
$\varphi(\alpha_n(\tilde t_{n,-}))$, $\varphi(\alpha_n(t_{n,-}))$, $\varphi(\alpha_n(t_{n,+}))$, $\varphi(\alpha_n(\tilde t_{n,+}))$,
see Figure \ref{fig_Sn}.

%%%%%%%%%%%%%%%%%%%%%%%%%%%%%%%%%%%%%%%
\begin{figure}[hbt]
\centering
\scalebox{0.3}{\includegraphics[clip]{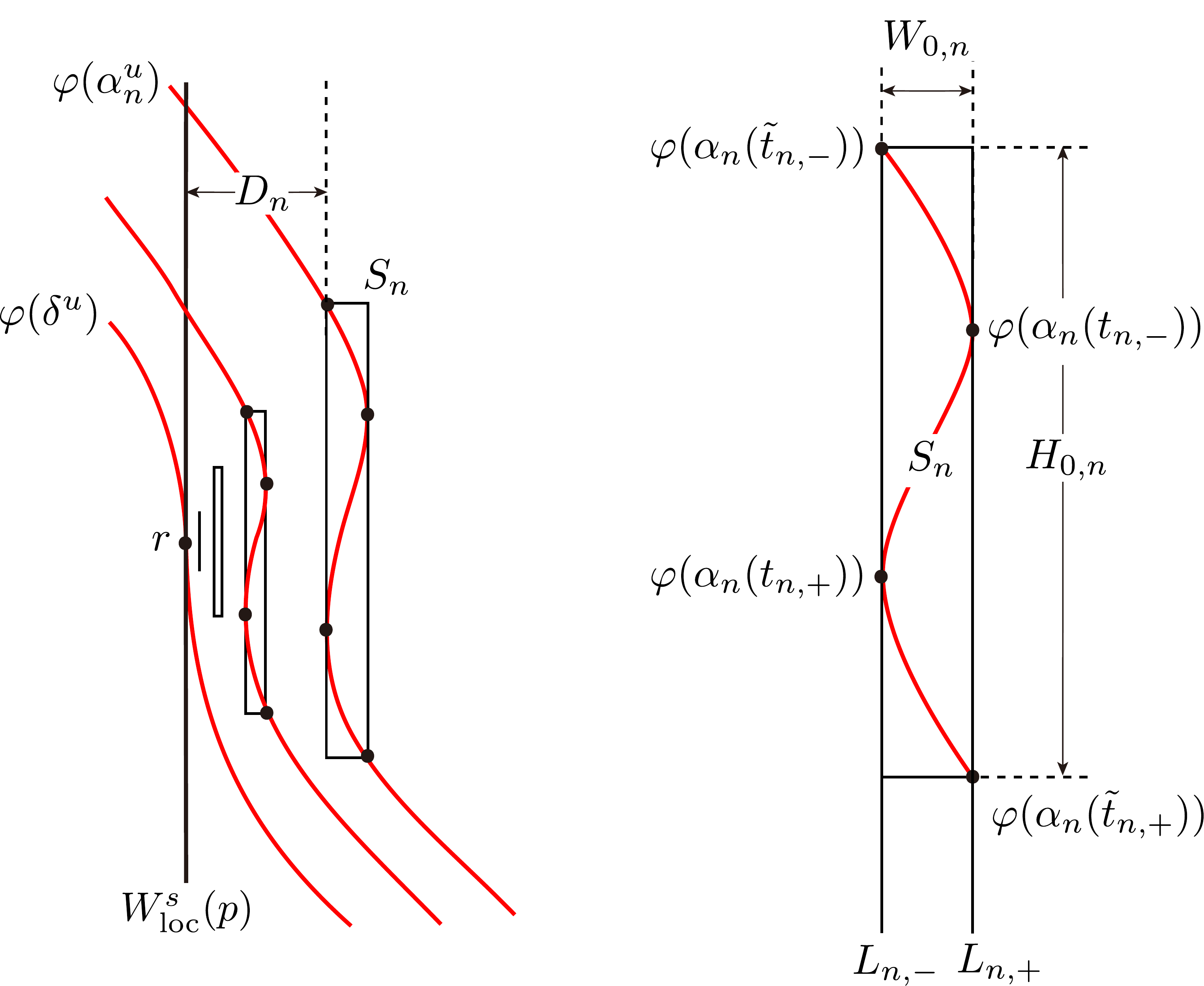}}
\caption{}
\label{fig_Sn}
\end{figure}
%%%%%%%%%%%%%%%%%%%%%%%%%%%%%%%%%%%%%%%

Now we will estimate the size of $S_n$.
Let $D_n$ be the distance between $S_n$ and $W_{\mathrm{loc}}^s(p)$.
Then
\begin{equation}\label{eqn_SD}
\begin{split}
D_n&\approx 
a\tilde y_n(t_{n,+})+bt_{n,+}\tilde y_n(t_{n,+})+ct_{n,+}^3\\
&\approx
az_0\lambda^n+bz_0\sqrt{\frac{-bz_0}{3c}}\lambda^{\frac32 n}-\frac{bz_0}{3}\sqrt{\frac{-bz_0}{3c}}\lambda^{\frac{3}{2}n} \sim \lambda^n.
\end{split}
\end{equation}
By \eqref{eqn_varphi}, the width $W_{0,n}$ of $S_n$ 
is represented as  
\begin{align*}
W_{0,n}&\approx (a\tilde y_n(t_{n,-})+bt_{n,-}\tilde y_n(t_{n,-})+ct_{n,-}^3)-(a\tilde y_n(t_{n,+})+bt_{n,+}\tilde y_n(t_{n,+})+ct_{n,+}^3)\\
&=a(\tilde y_n(t_{n,-})-\tilde y_n(t_{n,+}))+b(t_{n,-}\tilde y_n(t_{n,-})-t_{n,+}\tilde y_n(t_{n,+}))+c(t_{n,-}^3-t_{n,+}^3).
\end{align*}
It follows from Mean Value Theorem together with \eqref{eqn_y_n'} that
$$|\tilde y_n(t_{n,-})-\tilde y_n(t_{n,+})|\precsim \lambda^n\mu^{-n}|t_{n,-} -t_{n,+}|\sim \lambda^{\frac{3}{2}n}\mu^{-n}.$$
Moreover, by \eqref{eqn_tn+-}, we have
\begin{align*}
c(t_{n,-}^3-t_{n,+}^3)&\approx  
c\left(t_{n,-}\left(\frac{-b\tilde y_n(t_{n,-})}{3c}\right)-t_{n,+}\left(\frac{-b\tilde y_n(t_{n,+})}{3c}\right)\right)\\
&=-\frac{b}3(t_{n,-}\tilde y_n(t_{n,-})-t_{n,+}\tilde y_n(t_{n,+})).
\end{align*}
Since 
\begin{align*}
t_{n,-}\tilde y_n(t_{n,-})-t_{n,+}\tilde y_n(t_{n,+})&=(t_{n,-}-t_{n,+})\tilde y_n(t_{n,-})+t_{n,+}(\tilde y_n(t_{n,-})-y_n(t_{n,+}))\\
&\approx -\sqrt{\frac{-bz_0}{3c}}\lambda^{\frac{n}2}\cdot z_0\lambda^n
+O\left(\lambda^{\frac{n}2}\cdot \lambda^{\frac{3}{2}n}\mu^{-n}\right) \sim -\lambda^{\frac{3}{2}n},
\end{align*}
we have
\begin{equation}\label{eqn_SW}
\begin{split}
W_{0,n}\approx O(\lambda^{\frac{3}{2}n} \mu^{-n})+\frac{2b}{3} (t_{n,-}\tilde y_n(t_{n,-})-t_{n,+}\tilde y_n(t_{n,+})) \sim \lambda^{\frac{3}{2}n}.
\end{split}
\end{equation}

Next we estimate the height $H_{0,n}$ of $S_n$.
For that, we estimate $W_{0,n}$ again by using $\tilde t_{n,+}$ and $t_{n,+}$ instead of 
$t_{n,-}$ and $t_{n,+}$.
Since $\tilde t_{n,+}>t_{n,+}$, one can set $\tilde t_{n,+}=t_{n,+}+\rho_n\lambda^{\frac{n}2}$ for some $\rho_n>0$.
\begin{align*}
W_{0,n}&\approx 
a(\tilde y_n(\tilde t_{n,+})-\tilde y_n(t_{n,+}))+
b(\tilde t_{n,+}\tilde y_n(\tilde t_{n,+})-t_{n,+}\tilde y_n(t_{n,+}))
+c(\tilde t_{n,+}^3-\tilde t_{n,-}^3)\\
&=(a+b\tilde t_{n,+})(\tilde y_n(\tilde t_{n,+})-\tilde y_n(t_{n,+}))
+b(\tilde t_{n,+}-t_{n,+})\tilde y_n(t_{n,+})+c(\tilde t_{n,+}^3-t_{n,-}^3).
\end{align*}
Again by Mean Value Theorem together with \eqref{eqn_y_n'},
$$|\tilde y_n(\tilde t_{n,+})-\tilde y_n(t_{n,+})| \precsim \lambda^n\mu^{-n}\cdot \rho_n\lambda^{\frac{n}2} 
=\rho_n\lambda^{\frac{3}{2}n}\mu^{-n}.$$
Moreover, we have
\begin{align*}
(\tilde t_{n,+}-t_{n,+})\tilde y_n(t_{n,+}) &\sim \rho_n\lambda^{\frac{n}2}\cdot \lambda^n =\rho_n\lambda^{\frac{3}{2}n}
\end{align*}
and
\begin{align*}
\tilde t_{n,+}^3-t_{n,+}^3&=3\rho_n^2\lambda^n t_{n,+}+3\rho_n\lambda^{\frac{n}2}t_{n,+}^2+\rho_n^3\lambda^{\frac{3}{2}n}\\
&\approx \left(3\rho_n\sqrt{\dfrac{-bz_0}{3c}}-\frac{3bz_0}{c}+\rho_n^2\right)\rho_n\lambda^{\frac{3}{2}n}.
\end{align*}
This shows that 
\begin{align*}
W_{0,n}&\sim \left( a\mu^{-n} + bz_0 + 3\rho_n\sqrt{\dfrac{-bz_0}{3c}}-\frac{3bz_0}{c}+\rho_n^2\right)
\rho_n\lambda^{\frac{3}{2}n}.
\end{align*}
Since $W_{0,n} \sim \lambda^{\frac{3}{2}n}$, it follows that $\rho_n \sim 1$ and hence 
$\tilde t_{n,+}\sim \lambda^{\frac{n}2}$.
Similarly $-\tilde t_{n,-}\sim \lambda^{\frac{n}2}$.
This implies that 
\begin{equation}
|\tilde t_{n,\pm}|\sim \lambda^{\frac{n}2}.
\end{equation}
Therefore we have
\begin{equation}\label{eqn_SH}
\begin{split}
H_{0,n}&=(1+d\tilde t_{n,-}+e\tilde y_n(\tilde t_{n,-})) - (1+d\tilde t_{n,+}+e\tilde y_n(\tilde t_{n,+}))\\
&=d(\tilde t_{n,-} - \tilde t_{n,+}) + e(\tilde y_n(\tilde t_{n,-})-\tilde y_n(\tilde t_{n,+})) \sim \lambda^{\frac{n}{2}}+O(\lambda^{\frac{3}{2}n}\mu^{-n}) \sim \lambda^{\frac{n}{2}}.
\end{split}
\end{equation}
In particular, $\{S_n\}$ is a sequence  of rectangles converging to the cubic tangency $r$.

\section{Slope Lemma}\label{S_SLPL}

Let $\boldsymbol{v}=\begin{bmatrix}u\\v\end{bmatrix}\in T_{\boldsymbol{x}}(M)$ be a tangent vector at $\boldsymbol{x}\in U(p)$ 
with $u\neq 0$.
Then we say that $|vu^{-1}|$ is the (absolute) \emph{slope} of $\boldsymbol{v}$ and 
denote it by $\mathrm{Slope}(\boldsymbol{v})$.

Consider any tangent vector $\boldsymbol{v}_0=\begin{bmatrix}1\\ \delta\end{bmatrix}\in T_{\boldsymbol{x}}(M)$ at 
$\boldsymbol{x}=(x+1,y)\in R_\varepsilon$ 
with $|\delta|\leq \varepsilon^{\frac{5}{2}}$.
We set $\boldsymbol{v}_0'=d\varphi_{(x+1,y)}(\boldsymbol{v}_0)$ and $\boldsymbol{v}_1=df^{u_0}_{\varphi(x+1,y)} (\boldsymbol{v}_0')$.
By \eqref{eqn_dphi}, 
$$\mathrm{Slope}(\boldsymbol{v}_0')\approx \frac{|d+e\delta|}{|3cx^2+a\delta|}.$$
Since $\varepsilon \le x$ and $|\delta|\leq \varepsilon^{\frac{5}{2}}$, 
\begin{align*}
\mathrm{Slope}(\boldsymbol{v}_0') &\approx \frac{|d+e\delta|}{|3cx^2+a\delta|} \leq \frac{|d|+|e\delta|}{|3cx^2|-|a\delta|} \leq \frac{|d|+|e\varepsilon^{\frac{5}{2}}|}{|3c\varepsilon^2|-|a\varepsilon^{\frac{5}{2}}|}\\
& = \frac{|d|+|e\varepsilon^{\frac{5}{2}}|}{|3c|-|a\varepsilon^{\frac{1}{2}}|}\varepsilon^{-2} = \frac{|d|+|e\varepsilon^{\frac{5}{2}}|}{|3c|-|a\varepsilon^{\frac{1}{2}}|}\varepsilon^{\frac{1}{2}}\cdot \varepsilon^{-\frac{5}{2}}.
\end{align*}
By taking $\varepsilon_1>0$ sufficiently small, for any $0 < \varepsilon \le \varepsilon_1$, we have
\begin{align*}
\mathrm{Slope}(\boldsymbol{v}_0') 
\le 2\frac{|d|+|e\varepsilon^{\frac{5}{2}}|}{|3c|-|a\varepsilon^{\frac{1}{2}}|}\varepsilon^{\frac{1}{2}}\cdot \varepsilon^{-\frac{5}{2}} 
\le \frac{|3d|}{|2c|}\varepsilon^{\frac{1}{2}}\cdot \varepsilon^{-\frac{5}{2}} 
\le 1 \cdot \varepsilon^{-\frac{5}{2}}.
\end{align*}
Then,by \eqref{eqn_slow1} and \eqref{eqn_muu0}, we have
\begin{align*}
\mathrm{Slope}(\boldsymbol{v}_1)&=\mathrm{Slope}(\boldsymbol{v}_0')\lambda^{u_0}\mu^{-u_0}
\leq \varepsilon^{-\frac52}\lambda^{u_0}\mu^{-u_0}
\leq \varepsilon^{-\frac52}\mu^{-\frac52 u_0}
\leq \varepsilon^{-\frac52}\varepsilon^5=\varepsilon^{\frac52}.
\end{align*}
Thus we get the following lemma.
See Figure \ref{fig_slope}.

%%%%%%%%%%%%%%%%%%%%%%%%%%%%%%%%%%%%%%%
\begin{figure}[hbt]
\centering
\scalebox{0.3}{\includegraphics[clip]{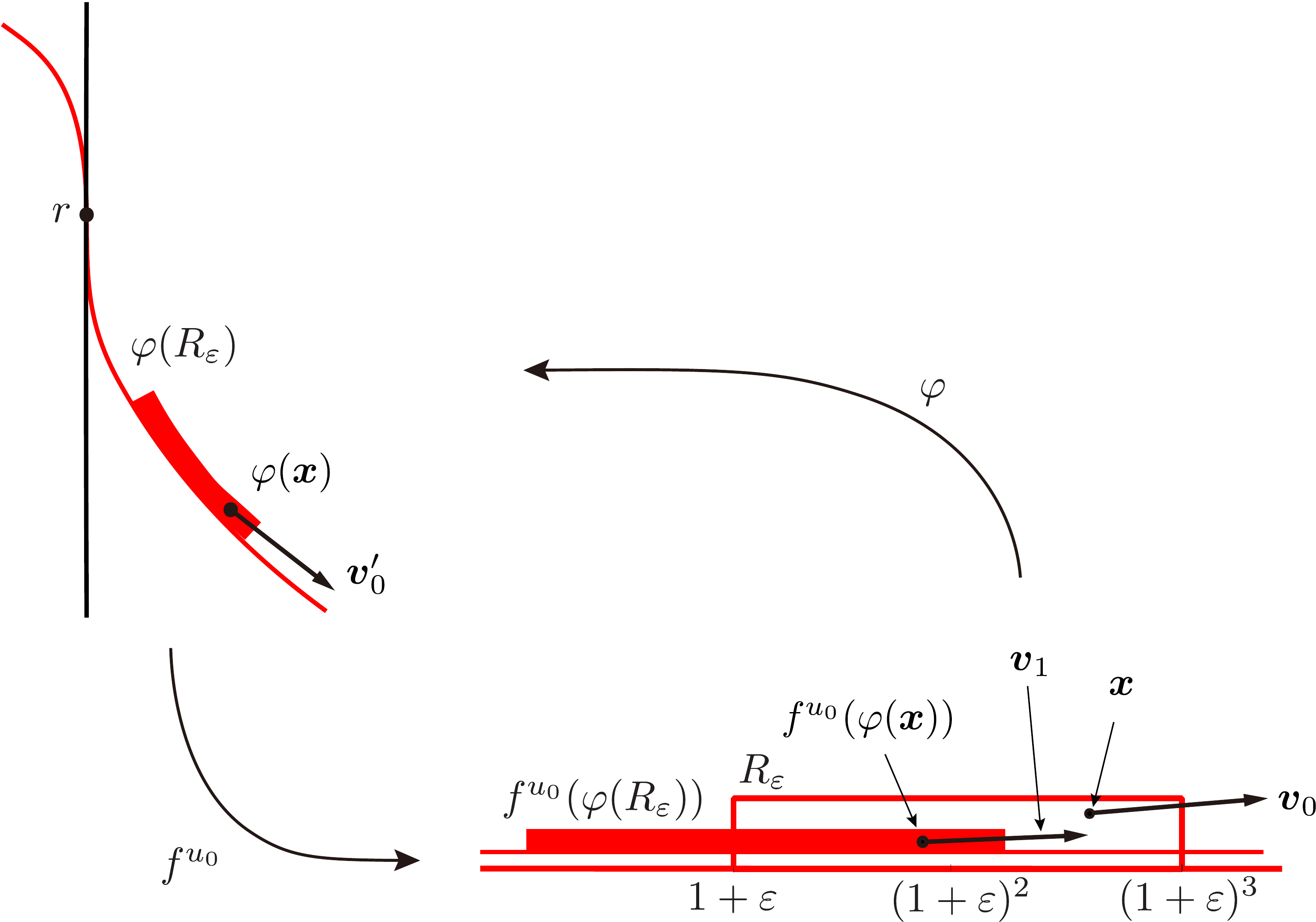}}
\caption{}
\label{fig_slope}
\end{figure}
%%%%%%%%%%%%%%%%%%%%%%%%%%%%%%%%%%%%%%%

\begin{lemma}[Slope Lemma I]\label{l_slope}
Suppose that $f$ satisfies the conditions \eqref{eqn_slow1}.
Then there exists a constant $\varepsilon_1>0$ such that, 
if $\varepsilon \in (0, \varepsilon_1]$, 
then
\begin{equation}\label{eqn_slow2}
\mathrm{Slope}(\boldsymbol{v}_0')\leq \varepsilon^{-\frac{5}{2}}\quad\text{and}\quad \mathrm{Slope}(\boldsymbol{v}_1)\leq \varepsilon^{\frac{5}{2}}
\end{equation}
for any tangent vector $\boldsymbol{v}_0\in T_{\boldsymbol{x}}(M)$ at 
$\boldsymbol{x}=(x+1,y)\in R_\varepsilon$ with $\mathrm{Slope}(\boldsymbol{v}_0)\leq \varepsilon^{\frac{5}{2}}$.
\end{lemma}

Fix a sufficiently small $s>0$ and set 
$\mathrm{pr}_x(S_n)=[s^-_n,s^+_n]$ for $n\in \mathbb{N}$.
If $n$ is sufficiently large, then $[s_n^-,s_n^+]\subset (0,s]$.
Let $\beta_n^u(s)$ be the component of $\varphi(\alpha_n^u) \cap \mathrm{pr}_x^{-1}((0,s])$ 
containing $\varphi(\alpha_n([\tilde t_{n,-},\tilde t_{n,+}]))$.
For any $\boldsymbol{x}\in \beta_n^u(s)$, let $j_n(\boldsymbol{x})$ be a positive integer 
such that 
$f^j(\boldsymbol{x})\in U(p)$ for $j=1,\dots,j_n(\boldsymbol{x})$ and 
$\mathrm{pr}_x(f^{j_n(\boldsymbol{x})}(\boldsymbol{x})) \in [1+\varepsilon,(1+\varepsilon)^3]$.
For any $\varepsilon>0$, one can take $s$ 
so that $\mathrm{pr}_x(f^{j_n(\boldsymbol{x})}(\boldsymbol{x})) \in R_\varepsilon$ for any $\boldsymbol{x} \in \beta_n^u(s)$.
Let $\boldsymbol{v}(\boldsymbol{x})$ be a unit vector tangent to $\beta_n^u(s)$ at $\boldsymbol{x}$.

The following result is applied to $f_1$ in the proof of Theorem \ref{thm_A}.

\begin{lemma}[Slope Lemma II]\label{l_slopef1}
Let $\varepsilon_1$ be the constant given in Lemma \ref{l_slope}.
For any $\varepsilon \in (0, \varepsilon_1]$, there exist $s>0$ and $n_0\in \mathbb{N}$ 
such that
$$\mathrm{Slope}(df_{\boldsymbol{x}}^{j_n(\boldsymbol{x})}(\boldsymbol{v}(\boldsymbol{x})))<\varepsilon^{\frac{5}{2}}$$
if $n\geq n_0$ and $\boldsymbol{x}\in \beta_n^u(s)\setminus S_n$.
\end{lemma}

\begin{proof}
We only consider the case where $\boldsymbol{x}$ is an element of $\beta_n^u(s)\setminus S_n$ 
with $\mathrm{pr}_x(\boldsymbol{x})\geq s_n^+$.
Then $t\geq \tilde t_{n,+}$ holds if $\varphi(\alpha_n(t))=\boldsymbol{x}$. 
The proof in the case of $\mathrm{pr}_x(\boldsymbol{x})\leq s_n^-$ is done similarly.
Since $\rho_n\sim 1$ and $\tilde t_{n,+}=t_{n,+}+\rho_n\lambda^{\frac{n}{2}}$, 
$t-t_{n,+}\geq \tilde t_{n,+}-t_{n,+}\sim \lambda^{\frac{n}2}.$
This implies that
\begin{equation}\label{eqn_t-t}
t^2-t_{n,+}^2\sim t^2\succsim \lambda^n.
\end{equation}
In fact, if $t-t_{n,+}\geq \frac{t}2$, then 
$t^2-t_{n,+}^2=(t-t_{n,+})(t+t_{n,+})>\frac{t^2}2$ and hence \eqref{eqn_t-t} holds.
On the other hand, if $t-t_{n,+}\leq \frac{t}2$, then $t\leq 2t_{n,+}$ and so $t\sim \lambda^{\frac{n}2}$.
It follows that $t+t_{n,+}\sim \lambda^{\frac{n}2}$ and $t-t_{n,+}\sim \lambda^{\frac{n}2}$.
Then $t^2-t_{n,+}^2\sim \lambda^n\sim t^2$.
Thus \eqref{eqn_t-t} holds.

We set $\xi_n(t)=\mathrm{pr}_x(\boldsymbol{x})=\mathrm{pr}_x(\varphi(\alpha_n(t)))$.
By \eqref{eqn_va},
\begin{equation}\label{eqn_xixi}
\begin{split}
\xi_n(t)&=a\tilde y_n(t)+bt\tilde y_n(t)+ct^3+\mathrm{h.o.t.},\\
\xi_n'(t)&=
a\tilde y_n'(t)+b\tilde y_n(t)+bt\tilde y_n'(t)+3ct^2+\mathrm{h.o.t.}
\end{split}
\end{equation}
From the definition of $j_n(\boldsymbol{x})$, 
$$\mu^{j_n(\boldsymbol{x})} \xi_n(t)= \mu^{j_n(\boldsymbol{x})} \mathrm{pr}_x(\boldsymbol{x})= \mathrm{pr}_x(f^{j_n(\boldsymbol{x})}(\boldsymbol{x})) \in [1+\varepsilon,(1+\varepsilon)^3].$$
This implies that $\mu^{j_n(\boldsymbol{x})}\xi_n(t)\sim 1$.
We note that $\xi_n'(t_{n,+})=0$.
By Mean Value Theorem, $\tilde y_n(t)-\tilde y_n(t_{n,+})=\tilde y_n'(c)(t-t_{n,+})$ for some 
$t_{n,+}<c<t$.
From this fact together with \eqref{eqn_psi}, \eqref{eqn_y_n'}, \eqref{eqn_t-t} and 
\eqref{eqn_xixi}, we know that 
$$\xi_n'(t)=\xi_n'(t)-\xi_n'(t_{n,+})\sim t^2-t_{n,+}^2\sim t^2.$$
By \eqref{eqn_va'}, $\mathrm{Slope}(\boldsymbol{v}(\boldsymbol{x}))\sim t^{-2}$.
Hence we have
\begin{equation}\label{eqn_slopedf}
\mathrm{Slope}(df_{\boldsymbol{x}}^{j_n(\boldsymbol{x})}(\boldsymbol{v}(\boldsymbol{x}))) 
=\mathrm{Slope}(\boldsymbol{v}(\boldsymbol{x})) \cdot \frac{\lambda^{j_n(\boldsymbol{x})}}{\mu^{j_n(\boldsymbol{x})}} 
\sim t^{-2}\lambda^{j_n(\boldsymbol{x})}\xi_n(t).
\end{equation}
Now we need to consider the following two cases.

\noindent{\it Case 1.}
$ct^3\leq a\tilde y_n(t)$.
By \eqref{eqn_xixi}, $\xi_n(t)\sim \lambda^n$.
Since $t^{-2}\precsim \lambda^{-n}$ by $t\succsim \lambda^{\frac{n}2}$, 
it follows from \eqref{eqn_slopedf} that
$$
\mathrm{Slope}(df_{\boldsymbol{x}}^{j_n(\boldsymbol{x})}(\boldsymbol{v}(\boldsymbol{x}))) \precsim \lambda^{-n}\lambda^{j_n(\boldsymbol{x})}\lambda^n=
\lambda^{j_n(\boldsymbol{x})}.$$

\noindent{\it Case 2.}
$ct^3\geq a\tilde y_n(t)$.
Again by \eqref{eqn_xixi}, we have $\xi_n(t)\sim t^3$.
Then, by \eqref{eqn_slopedf}, 
$$
\mathrm{Slope}(df_{\boldsymbol{x}}^{j_n(\boldsymbol{x})}(\boldsymbol{v}(\boldsymbol{x}))) 
\sim t^{-2}\lambda^{j_n(\boldsymbol{x})}t^3=t \lambda^{j_n(\boldsymbol{x})}\precsim \lambda^{j_n(\boldsymbol{x})}.
$$

Let $n_0(s)$ be the minimum positive integer with $s_{n_0(s)}^+ < s$.
Since $\lim_{s\to +0}n_0(s)=\infty$, one can take $s=s(\varepsilon)>0$ such that    
our desired inequality holds for any $\boldsymbol{x}\in \beta_n^u(s)\setminus S_n$.
\end{proof}

\section{Sequence of rectangle-like boxes}\label{S_Bbox}

Now we will define a sequence $\{B_{k,n}\}_{k=1}^\infty$ of rectangle-like boxes and 
estimate the sizes of them.

Recall that $\mathrm{pr}_x(S_n) = [s_n^-,s_n^+]$.
Let $i_n$ be the positive integer with $(1+\varepsilon)^2<\mu^{i_n}s_n^{+}\leq (1+\varepsilon)^3$. 
By \eqref{eqn_Rve}, $f^{i_n}(S_n)$ is contained in $R_\varepsilon$ for any sufficiently large $n$.
We set $f^{i_n}(S_n)=B_{1,n}=B_1$ for short.  
Since $s_n^{+} \sim \lambda^n$ by \eqref{eqn_SD} and \eqref{eqn_SW}, 
we have
\begin{equation}\label{eqn_innsim1}
\mu^{i_n}\lambda^n \sim 1.
\end{equation}
We denote the width and height of $B_1$ and 
the distance between $B_1$ and $W^u_{\mathrm{loc}}(p)$ by 
$W_{1,n}=W_1$, $H_{1,n}=H_1$ and $L_{1,n}=L_1$ respectively.
It follows from \eqref{eqn_SD}, \eqref{eqn_SW} and \eqref{eqn_SH} that 
\begin{equation}\label{eqn_W_1}
W_{1,n}\sim \lambda^{\frac32 n}\mu^{i_n}\sim \lambda^{\frac{n}2},\quad
H_{1,n}\sim \lambda^{\frac{n}2+i_n},\quad L_{1,n}\sim \lambda^{i_n}.
\end{equation}
Note that, for any sufficiently large $n$, $H_1\ll L_1\ll W_1$.
Consider a closed interval $\delta_1$ in $W_{\mathrm{loc}}^u(p)$ which is a 
small neighborhood of $\mathrm{pr}_x(B_1)$.

Let $v_i^{(1)}$, $e_i^{(1)}$ ($i=0,1,2,3$) be the vertices and edges of $B_1$ 
as illustrated in Figure \ref{fig_B1}\,(a).
%%%%%%%%%%%%%%%%%%%%%%%%%%%%%%%%%%%%%%%
\begin{figure}[hbt]
\centering
\scalebox{0.3}{\includegraphics[clip]{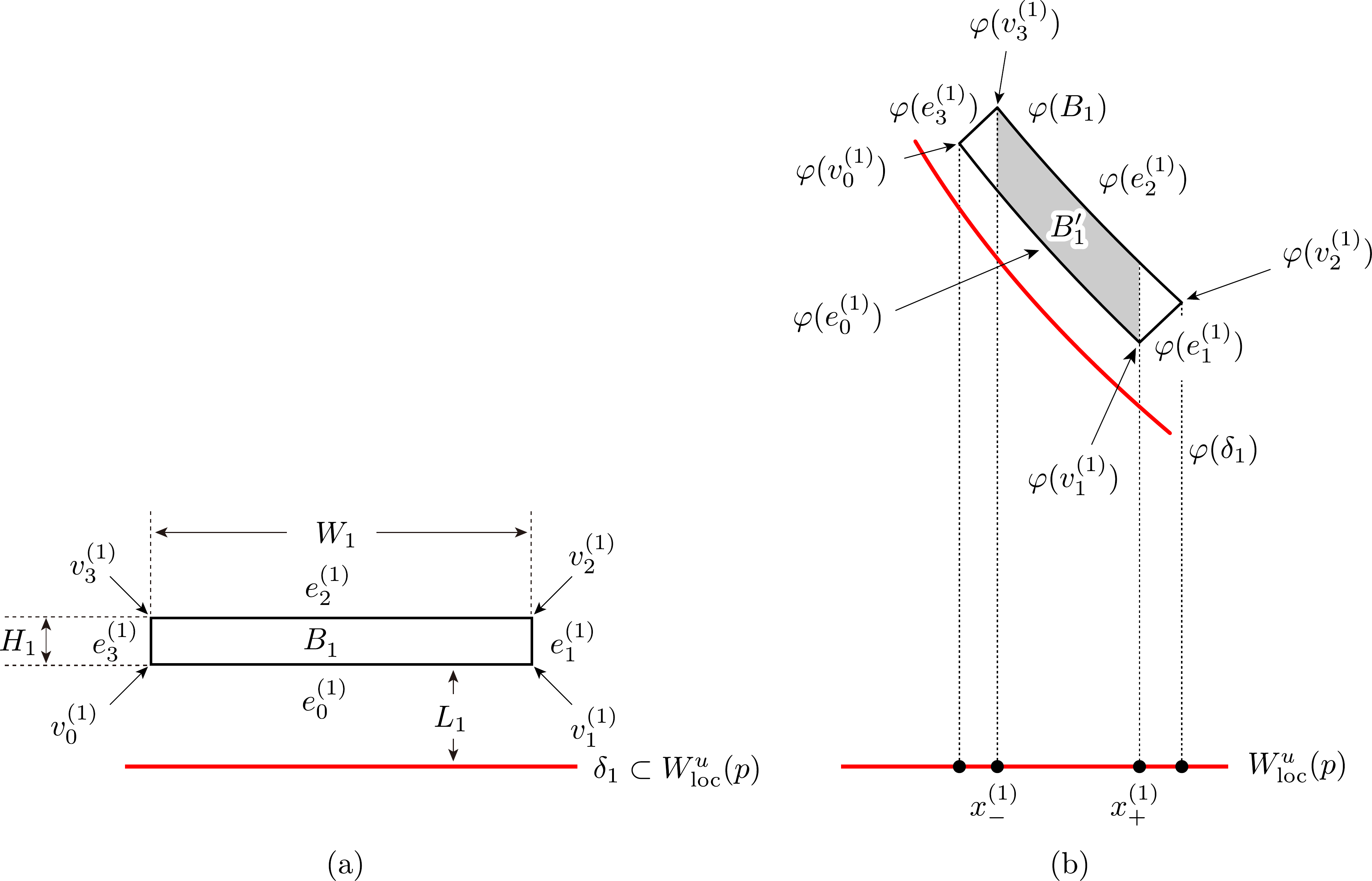}}
\caption{}
\label{fig_B1}
\end{figure}
%%%%%%%%%%%%%%%%%%%%%%%%%%%%%%%%%%%%%%%
We consider the image $\varphi(B_1)$.
By Lemma \ref{l_slope}, for $i=0,2$, 
$$\mathrm{diam}(\mathrm{pr}_x(\varphi(e_i^{(1)}))) \succsim \varepsilon^{\frac{5}{2}}W_1 \sim \varepsilon^{\frac{5}{2}}\lambda^{\frac{n}2}.$$
On the other hand, for $i=1,3$,
$$\mathrm{diam}(\mathrm{pr}_x(\varphi(e_i^{(1)}))) \precsim H_1 \sim \lambda^{\frac{n}2+i_n}.$$
Since $\lambda^{i_n}\varepsilon^{-\frac{5}{2}}$ can be supposed to be arbitrarily small 
for all sufficiently large $n$, 
\begin{equation}\label{eqn_x+x-}
\begin{split}
x_+^{(1)}-x_-^{(1)} &\succsim \varepsilon^{\frac{5}{2}}W_1-O(\lambda^{\frac{n}2+i_n}) \sim \varepsilon^{\frac{5}{2}}\lambda^{\frac{n}{2}}-O(\lambda^{\frac{n}{2}+i_n}) \\
&=\varepsilon^{\frac52}\lambda^{\frac{n}{2}}\left(1-\frac{O(\lambda^{i_n})}{\varepsilon^{\frac{5}{2}}} \right) \sim \varepsilon^{\frac{5}{2}}\lambda^{\frac{n}{2}}.
\end{split}
\end{equation}
where $x_+^{(1)}=\mathrm{pr}_x(\varphi(v_1^{(1)}))$ and $x_-^{(1)}=\mathrm{pr}_x(\varphi(v_3^{(1)}))$, 
see Figure \ref{fig_B1}\,(b).
Let $B_{1}'$ be the intersection $\mathrm{pr}_x^{-1}\bigl([x_-^{(1)}, x_+^{(1)}]\bigr)\cap \varphi(B_{1})$.
Any compact region in $U(p)$ like $B_{1}'$ is called a \emph{parallelogram-like box}.

Let $u_1$ be the positive integer with $(1+\varepsilon)^2 < \mu^{u_1}x_+^{(1)} \leq (1+\varepsilon)^3$.
By \eqref{eqn_Rve}, $f^{u_1}(B_1')$ is contained in $R_\varepsilon$ for any sufficiently large $n$. 
We denote $f^{u_1}(B_1')$ by $B_2$.
We call that any compact region in $U(p)$ like $B_{2}$ is a \emph{rectangle-like box}.

Let $B$ be either a parallelogram-like or rectangle-like box.
The \emph{horizontal width} of $B$ is the diameter of the interval $\mathrm{pr}_x(B)$.
The \emph{vertical height} of $B$ is the maximum of the lengths of $\eta(x_0)$ with $x_0\in \mathrm{pr}_x(B)$, where 
$\eta(x_0)$ is the intersection of $B$ and the vertical line $x=x_0$.
See Figure \ref{fig_vertical} in the case of $B=B_1'$. 
%%%%%%%%%%%%%%%%%%%%%%%%%%%%%%%%%%%%%%%
\begin{figure}[hbt]
\centering
\scalebox{0.3}{\includegraphics[clip]{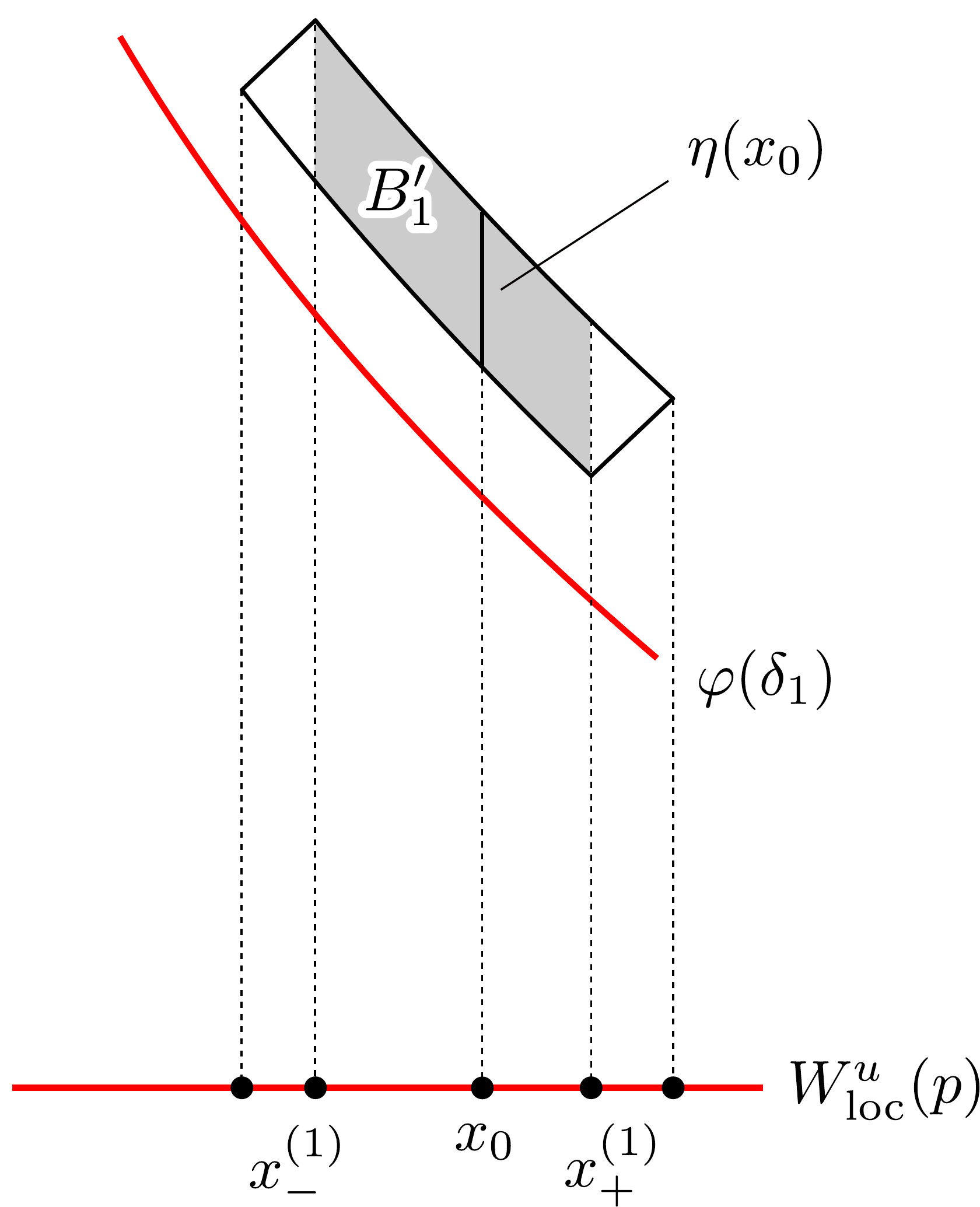}}
\caption{}
\label{fig_vertical}
\end{figure}
%%%%%%%%%%%%%%%%%%%%%%%%%%%%%%%%%%%%%%%
Suppose that $B$ is a rectangle-like box and $\delta$ is an almost horizontal arc in $U(q)$ with 
$B\cap \delta=\emptyset$ and $\mathrm{pr}_x(B)\subset \mathrm{pr}_x(\delta)$.
Then the \emph{vertical distance} between $B$ and $\delta$ is the maximum of $\sigma(x_1)$ with 
$x_1\in \mathrm{pr}_x(B)$, where $\sigma(x_1)$ is the length of the shortest segment in the vertical line $x=x_1$ connecting $B$ with $\delta$.

Let $\delta_2$ be a sub-arc of $f^{u_1}(\varphi(\delta_1))\subset W^u(p)$ such that 
$\mathrm{pr}_x(\delta_2)$ is a small neighborhood of $\mathrm{pr}_x(B_2)$ in $W_{\mathrm{loc}}^u(p)$.
We denote the horizontal width and vertical height of $B_2$ and the vertical distance between $B_2$ and $\delta_2$ by $W_2$, $H_2$ and $L_2$ respectively.
By \eqref{eqn_muu0}, \eqref{eqn_W_1} and \eqref{eqn_x+x-},
\begin{equation}\label{eqn_W2}
W_2=(x_+^{(1)}-x_-^{(1)})\mu^{u_1} \succsim \varepsilon^{\frac{5}{2}}\lambda^{\frac{n}2}\mu^{u_1} \geq 
\varepsilon^{-\frac{1}{2}}\tau_1^{-1}\lambda^{\frac{n}2} \sim \varepsilon^{-\frac{1}{2}}W_1.
\end{equation}

For any $x_0$ with $x_-^{(1)}\leq x_0\leq x_+^{(1)}$, 
$\eta(x_0)$ is a vertical segment connecting $\varphi(e_0^{(1)})$ with $\varphi(e_2^{(1)})$.
By this fact together with \eqref{eqn_slow2}, one can show the vertical height $H_1'$ of $B_1'$ satisfies 
$H_1' \precsim H_1\varepsilon^{-\frac{5}{2}}$.
It follows from \eqref{eqn_slow1} and \eqref{eqn_muu0} that 
\begin{equation}\label{eqn_H2}
H_2=\lambda^{u_1}H_1' \precsim \mu^{-\frac{3}{2}u_1} \varepsilon^{-\frac{5}{2}}H_1 < (\varepsilon^2)^{\frac{3}{2}}\varepsilon^{-\frac{5}{2}}H_1 =\varepsilon^{\frac12}H_1.
\end{equation}
Let $L_2$ be the vertical distance between $B_2$ and $\delta_2$.
By using an argument similar to that for the estimation \eqref{eqn_H2}, 
we have 
\begin{equation}\label{eqn_L2}
L_2\precsim \varepsilon^{\frac12}L_1.
\end{equation}

The following lemma is obtained immediately from \eqref{eqn_W2}, \eqref{eqn_H2} and \eqref{eqn_L2}.

\begin{lemma}\label{l_slow3}
Let $\varepsilon_1>0$ be the constant given in Lemma \ref{l_slope}.
Then there exists a constant $\varepsilon_0 \in (0, \varepsilon_1]$ 
such that, for any $\varepsilon \in (0, \varepsilon_0)$, the inequalities
\begin{equation*}
W_2\geq 10W_1,\quad H_2\leq 10^{-1}H_1\quad\mbox{and}\quad L_2\leq 10^{-1}L_1
\end{equation*}
hold.
\end{lemma}

If $\mu=1+\varepsilon$ for an $\varepsilon \in (0, \varepsilon_0)$, then we say that $f$ satisfies  
the \emph{small expanding conditions} at $p$.

We repeat the process as above.
Let $B_2'$ be the subset of $\varphi(B_2)$ cobounded by the vertical lines $x=x_-^{(2)}$ and 
$x=x_+^{(2)}$ passing through two of the four vertices of $\varphi(B_2)$ and satisfying  
$[x_-^{(2)}, x_+^{(2)}] \subset \mathrm{Int}(\mathrm{pr}_x(\varphi(B_2)))$.
Let $u_2$ be the positive integer with $(1+\varepsilon)^2 < \mu^{u_2}x_+^{(2)} \leq (1+\varepsilon)^3$ 
and $f^{u_2}(B_2') \subset R_\varepsilon$ for sufficient large $n \in \mathbb{N}$. 
Set $B_3=f^{u_2}(B_2')$.
Let $\delta_3$ be a sub-arc of $f^{u_2}(\varphi(\delta_2))$ such that $\mathrm{pr}_x(\delta_3)$ 
is a small neighborhood of $\mathrm{pr}_x(B_3)$ in $W_{\mathrm{loc}}^u(p)$.
We denote the horizontal width and vertical height of $B_3$ and the vertical distance between $B_3$ and $\delta_3$ 
by $W_3$, $H_3$ and $L_3$ respectively.

The objects $B_{k}'$, $u_{k}$, $B_{k+1}$, $\delta_{k}$, $W_{k+1}$, $H_{k+1}$, $L_{k+1}$ $(k=3,4,5,\dots)$ 
are defined inductively if 
\begin{equation}\label{eqn_prB}
B_j \subset R_\varepsilon
\end{equation}
for $j=1,2,\dots,k$.

The top and bottom sides of the rectangle $B_1$ are horizontal 
and $\gamma_1=B_1 \cap W^u(p)$ consists of three proper arcs in $B_1$.
By Slope Lemma I (Lemma \ref{l_slope}), for $k=2,3,\cdots$, 
the top and bottom sides of the rectangle-like box $B_k$ are almost horizontal 
and $\gamma_k=B_k \cap W^u(p)$ consists of three proper arcs in $B_k$.
See Figure \ref{fig_3on1}.
Thus we have the following lemma.

\begin{lemma}\label{l_slow3'}
Let $\varepsilon_0 >0$ be the constant given in Lemma \ref{l_slow3}.
For any $\varepsilon \in (0, \varepsilon_0]$, there exists the maximum integer $k_0=k_0(\varepsilon,n)$ satisfying \eqref{eqn_prB}.
Moreover, 
\begin{equation}\label{eqn_slow3}
W_{k+1}\geq 10W_k,\quad H_{k+1}\leq 10^{-1}H_k\quad\mbox{and}\quad L_{k+1}\leq 10^{-1}L_k
\end{equation}
hold for any $k=1,2,\dots,k_0$.
\end{lemma}
See Figure \ref{fig_k0} for the situation of Lemma \ref{l_slow3'}.
%%%%%%%%%%%%%%%%%%%%%%%%%%%%%%%%%%%%%%%
\begin{figure}[hbt]
\centering
\scalebox{0.3}{\includegraphics[clip]{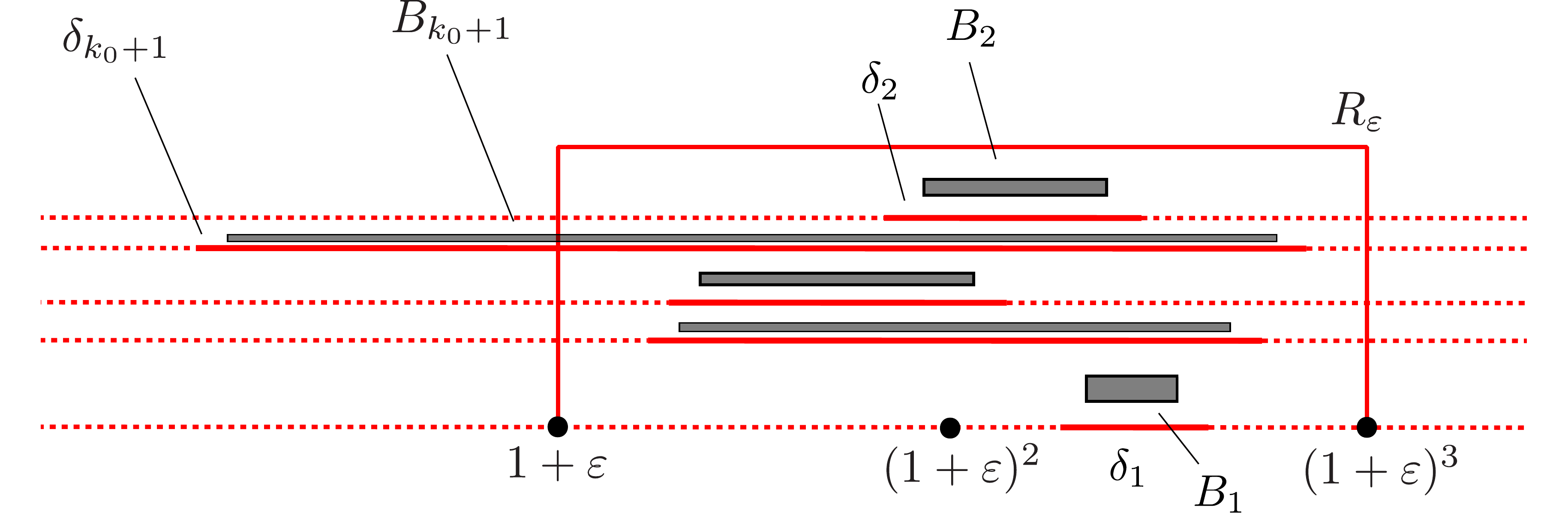}}
\caption{}
\label{fig_k0}
\end{figure}
%%%%%%%%%%%%%%%%%%%%%%%%%%%%%%%%%%%%%%%
We note that, since $W_1=W_{1,n}\sim \lambda^{\frac{n}2}$ by \eqref{eqn_W_1}, 
$\lim_{n\to\infty}k_0(\varepsilon,n)=\infty$ for a fixed $\varepsilon$ with $0<\varepsilon\leq \varepsilon_0$.

\section{Intersection Lemma}\label{S_ISL}

Recall that $\mathcal{C}$ is the codimension two submanifold of $\mathrm{Diff}^3(M)$ defined in 
Section \ref{S_pre}.
Let $f_0$, $f_1$ be elements of $\mathcal{C}$ satisfying the conditions \eqref{A1}--\eqref{A3} in 
Theorem \ref{thm_A}.
In particular, $\varepsilon>0$ is taken so that Slope Lemmas I and II 
(Lemmas \ref{l_slope} and \ref{l_slopef1}) hold.
Moreover, we suppose that $f_0$, $f_1$ satisfy the condition \eqref{eqn_adaptable_0}, 
which is one of the adaptable cases given in Section \ref{S_adaptable}.

From now on, we set $f_0=f$ and use the notations in Sections \ref{S_pre}--\ref{S_Bbox}.
Here the subscription `0' is omitted from the notations.
For example, $\lambda_0=\lambda$, $\mu_0=\mu$, $p_0=p$, $q_0=q$ and so on.
We also set $f_1=\overline f$ and represent the notations for $\overline f$ by 
adding bars to the corresponding notations for $f$, e.g.\ 
$\overline \lambda$, $\overline \mu$, $\overline p$, $\overline q$, $\overline m_0$, $\overline S_n$, $\overline B_k$, $\overline W_k$ and so on.

Let $h:M\to M$ be a homeomorphism with $\overline f=h\circ f\circ h^{-1}$.
Here we note that $h(r)$ is not necessarily equal to $\overline r$.
In fact, $h(r)=\overline r$ if and only if $m_0=\overline m_0$ or equivalently $\overline \varphi=h\circ \varphi\circ h^{-1}$.
We may assume that $m_0\leq \overline m_0$ if necessary replacing $f$ and $\overline f$.
Then $h(f^{\overline m_0-m_0}(r))=\overline r$.
Since the constants appeared in \eqref{eqn_varphi} depend on the coordinate on $U(p)$, 
one can not replace the coordinates on $U(p)$ or $U(\overline p)$ 
so as to satisfy $h(r)=\overline r$.

For any $C^1$-arc $\alpha$ in $U(p)$, the union of the end points of $\alpha$ is denoted by $\partial \alpha$.
When any vector tangent to $\alpha$ is not vertical, 
the maximum $\mathrm{Slope}(\alpha)$ of $\mathrm{Slope}(\boldsymbol{v}(\boldsymbol{x}))$ for vectors $\boldsymbol{v}(\boldsymbol{x})$ tangent to $\alpha$ at $\boldsymbol{x} \in \alpha$ is well defined.

If $s>0$ is small enough, then $\gamma_{0,n}'=\beta^u_n(s) \cap S_n$ is equal to $\alpha^u_n \cap S_n$ 
for any sufficiently large $n \in \mathbb{N}$.

The following is a key lemma for the proof of Theorem \ref{thm_A}. 

\begin{lemma}[Intersection Lemma]\label{l_intersection}
Let $\gamma'_{0,n}=\beta^u_n(s) \cap S_n$ and $\overline \gamma'_{0,n}=\overline\beta^u_n(\overline{s}) \cap \overline{S}_n$.
Then there exists an $n_0\in \mathbb{N}$ such that, 
for any $n \geq n_0$, 
\begin{equation}\label{eqn_hfgamma}
h(f^{\overline m_0-m_0}(\gamma_{0,n}')) \cap \overline \gamma_{0,n}' \neq \emptyset.
\end{equation}
\end{lemma}

\begin{proof}
We suppose that, for any $n_0\in \mathbb{N}$, there would exist $n>n_0$ such that 
$$\gamma_{0,n}'^{\,*} \cap \overline\gamma_{0,n}' = \emptyset,$$
where $\gamma_{0,n}'^{\,*}=h \circ f^{\overline{m}_0-m_0}(\gamma_{0,n}')$, 
and introduce a contradiction.

Recall that $i_n\in \mathbb{N}$ satisfies 
$(1+\varepsilon)^2 < f^{i_n}(s_n^+) \leq (1+\varepsilon)^3$ 
and 
$f^{i_n}(S_n) \subset R_{\varepsilon}$.
For short, we set 
$$\gamma_1=\gamma_{1,n}:=f^{i_n}(\gamma_{0,n}') \quad \text{and} \quad 
\gamma_1^*=\gamma_{1,n}^*:=h(\gamma_1).$$
Then $\gamma_1^*=\overline f^{\,i_n-(\overline m_0-m_0)}(\gamma_{0,n}^*)$.
Since $h(q)=\overline q$ and $\overline f=h \circ f \circ h^{-1}$, 
we have 
$h(1)=1, \,\,h(1+\varepsilon)=1+\overline\varepsilon, \,\,h\bigl((1+\varepsilon)^2\bigr)=(1+\overline\varepsilon)^2, \,\,h\bigl((1+\varepsilon)^3\bigr)=(1+\overline\varepsilon)^3$ 
and $1+\overline\varepsilon < \mathrm{pr}_x(\gamma_1^*) \leq (1+\overline\varepsilon)^3$.
Strictly, $\gamma_1^*$ may slightly exceed $R_{\overline\varepsilon}$.
Then we may rearrange our argument so that Lemmas \ref{l_slope} and \ref{l_slopef1} for $\overline f$ still hold 
if $\gamma_1^*$ is contained in a sufficiently small neighborhood of $R_{\overline\varepsilon}$.
Then, by applying Lemma \ref{l_slopef1} to $\overline{f}$, one can show that 
$\gamma_1^*$ is a sub-arc almost parallel to $\overline{\delta}_1 \subset W^u_{\mathrm{loc}}(\overline p)$ 
and $\mathrm{Slope}(\gamma_1^*) < \overline{\varepsilon}^{\frac{5}{2}}$ for any sufficiently large $n$. 

The intersection $\gamma_1'=\varphi(\gamma_1)\cap B_1'$ 
consists of mutually disjoint three arcs connecting the vertical sides of $B_1'$, see Figure \ref{fig_3arcs}.
%%%%%%%%%%%%%%%%%%%%%%%%%%%%%%%%%%%%%%%
\begin{figure}[hbt]
\centering
\scalebox{0.3}{\includegraphics[clip]{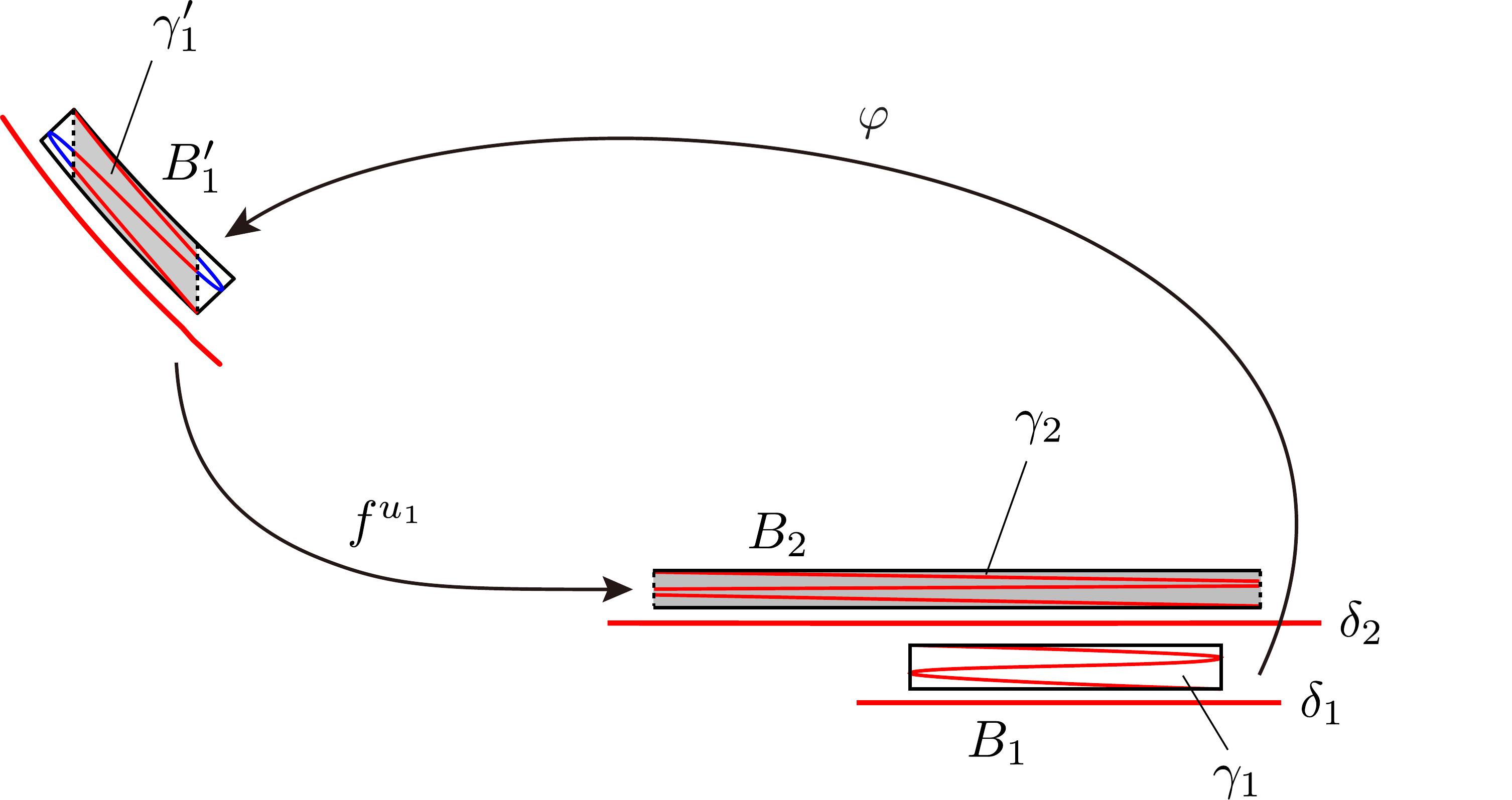}}
\caption{}
\label{fig_3arcs}
\end{figure}
%%%%%%%%%%%%%%%%%%%%%%%%%%%%%%%%%%%%%%%
We set 
$\gamma_2=f^{u_1}(\gamma_1')$ and $\gamma_2^*=h(\gamma_2)$.
Note that $\gamma_2$ is a disjoint union of three proper arcs in $B_2$ connecting the vertical sides of $B_2$.
Let $\widehat\gamma_2^*$ be the smallest arc in $W^u(\overline p)$ containing $\gamma_2^*$.
By applying Lemma \ref{l_slopef1} to $\overline{f}$, we have 
$\mathrm{Slope}(\widehat\gamma_2^*) < \overline{\varepsilon}^{\frac{5}{2}}$.
In particular, $\widehat\gamma_2^*$ is almost parallel to $\overline{\delta}_2$.
Repeating the same argument, 
one can have sequences $\{ \gamma_k\}$ satisfying the following conditions.
\begin{itemize}
\setlength{\leftskip}{-18pt}
\item
Each $\gamma_k$ is a disjoint union of three proper arcs in $B_k$ connecting the vertical sides of $B_k$.
\item
For each $\gamma_k^*=h(\gamma_k)$, the smallest arc $\widehat\gamma_k^*$ in $W^u(\overline p)$ containing $\gamma_k^*$ 
is almost parallel to $\overline{\delta}_k.$
\end{itemize}
See Figure \ref{fig_3on1}.
%%%%%%%%%%%%%%%%%%%%%%%%%%%%%%%%%%%%%%%
\begin{figure}[hbt]
\centering
\scalebox{0.3}{\includegraphics[clip]{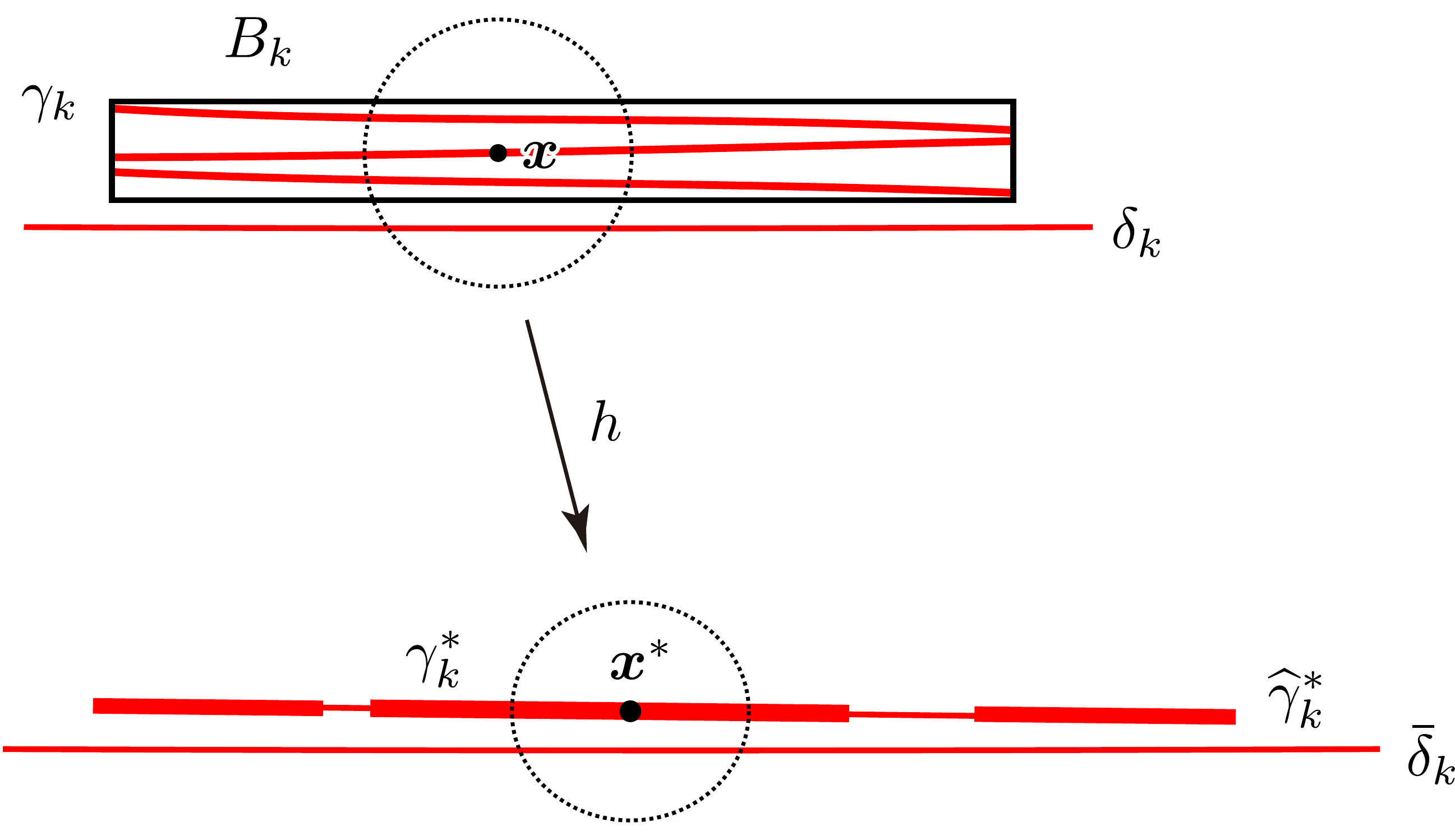}}
\caption{}
\label{fig_3on1}
\end{figure}
%%%%%%%%%%%%%%%%%%%%%%%%%%%%%%%%%%%%%%%

Take $\boldsymbol{x} \in \gamma_k$ arbitrarily and set $\boldsymbol{x}^* = h(\boldsymbol{x}) \in \gamma_k^*$.
Since $h$ is uniformly continuous on $R_{\varepsilon}$, 
for any $\overline{l}>0$, there exists $l>0$ independent of $\boldsymbol{x}$ such that 
$h \circ f^{\overline{m}_0-m_0}(N_{l}(\boldsymbol{x})) \subset N_{\overline{l}}(\boldsymbol{x}^*)$, 
where $N_{l}(\boldsymbol{x})$ is the $l$-neighborhood of $\boldsymbol{x}$ 
and $N_{\overline{l}}(\boldsymbol{x}^*)$ is the $\overline{l}$-neighborhood of $\boldsymbol{x}^*$ in $M$.
If $n$ is sufficiently large, then $N_{l}(\boldsymbol{x})$ must intersect the three arcs of $\gamma_k$.
However, $N_{\overline{l}}(\boldsymbol{x}^*)$ intersects only one arc of $\gamma_k^*$.
This gives a contradiction.  
Thus \eqref{eqn_hfgamma} holds for all sufficiently large $n$.
\end{proof}

\section{Proof of Theorem \ref{thm_A}}

Now we are ready to prove Theorem \ref{thm_A}.
The proof is done by using our Intersection Lemma (Lemma \ref{l_intersection}) together with arguments in \cite{pa, dm, po} and so on.
We only consider the case where both $f$ and $\overline f$ satisfy the condition \eqref{eqn_adaptable_0}, 
which belongs to Case $\mathrm{II}_{++}$ in Section \ref{S_adaptable}, and the small expanding 
condition at $p$ and $\overline p$ respectively.
The proof of any other adaptable case is done similarly.

\begin{proof}[Proof of (\ref{M1}) of Theorem \ref{thm_A}]
By Intersection Lemma (Lemma \ref{l_intersection}),
one can take $\overline{r}_n \in \overline{\gamma}_{0,n}' \cap h \circ f^{\overline{m}_0-m_0}(\gamma_{0,n}')$.
Since $\overline{r}_n$ converges to $\overline{r}$ as $n \to \infty$,
$r_n=(h \circ f^{\overline{m}_0-m_0})^{-1}(\overline{r}_n) \in \gamma_{0,n}'$ converges to $r$ as $n \to \infty$. 
See Figure \ref{fig_rn}.
%%%%%%%%%%%%%%%%%%%%%%%%%%%%%%%%%%%%%%%
\begin{figure}[hbt]
\centering
\scalebox{0.3}{\includegraphics[clip]{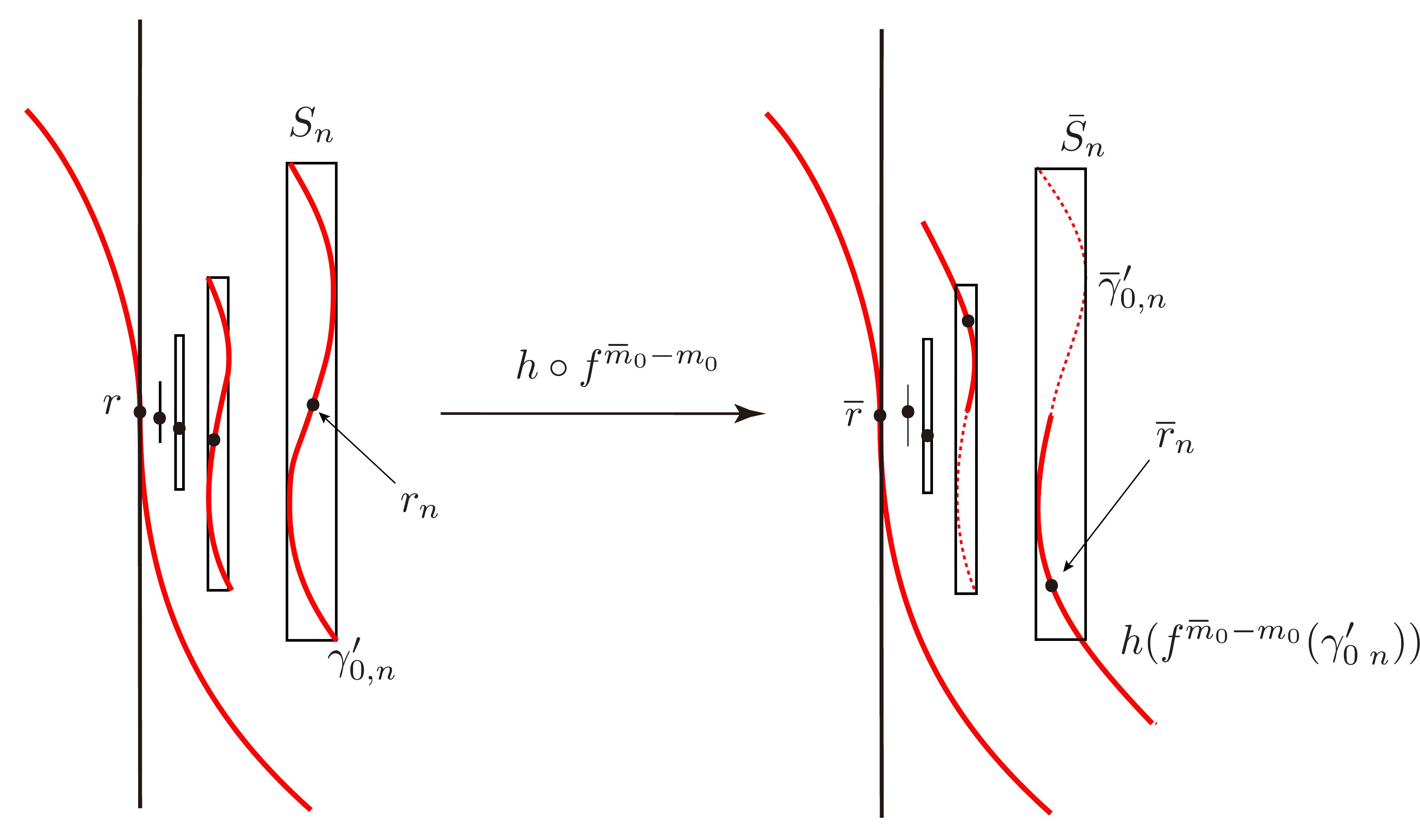}}
\caption{}
\label{fig_rn}
\end{figure}
%%%%%%%%%%%%%%%%%%%%%%%%%%%%%%%%%%%%%%%

Let $W_{\mathrm{loc},+}^u(p)$ be the component of $W_{\mathrm{loc}}^u(p)\setminus \{p\}$ containing $q$.
Take a fundamental domain $D$ for $f$ in $W^u_{\mathrm{loc},+}(p)$.
Then there exist subsequences $\{ r_{n(k)} \} \subset \{ r_n \}$, $\{ m(k) \}$ of $\mathbb{N}$ 
and $\boldsymbol{x}_0 \in D$ satisfying the following conditions.
\begin{itemize}
\setlength{\leftskip}{-18pt}
\item
$r_{n(k)}$ converges to $r$ as $k \to \infty$. 
\item
$\boldsymbol{x}_{n(k)}:=f^{m(k)}(r_{n(k)})$ converges to $\boldsymbol{x}_0=(x_0,0)$ as $k \to \infty$. 
\item 
$q_{n(k)}:=\varphi^{-1}(r_{n(k)})$ converges to $q$ as $k \to \infty$.
\end{itemize}
Then
\begin{align*}
 x_0&=\lim_{k\to\infty}\mathrm{pr}_x(\boldsymbol{x}_{n(k)}) = \lim_{k\to\infty}\mathrm{pr}_x(r_{n(k)})\mu^{m(k)}
=\lim_{k\to\infty}az_0\left(\lambda^{n(k)}+O\bigl(\lambda^{\frac{3}{2}n(k)}\bigr)\right)\mu^{m(k)}\\
&= \lim_{k\to\infty}az_0\lambda^{n(k)}\mu^{m(k)}.
\end{align*}
It follows that
$\lim_{k \to \infty} \lambda^{n(k)} \mu^{m(k)} = \dfrac{x_0}{az_0}$.
Then there exist constants $C_0$ and $C_1$ with $0<C_0<C_1$ and such that 
$$C_0<\lambda^{n(k)} \mu^{m(k)}<C_1$$
for any $k$.
Taking the logarithms of this inequalities,  
we have
$$\dfrac{\log C_0}{n(k)\log \mu} <  \dfrac{\log \lambda}{\log \mu} +\dfrac{m(k)}{n(k)} 
< \dfrac{\log C_1}{n(k)\log \mu}.$$
This shows that 
$\displaystyle \lim_{k \to \infty} \dfrac{m(k)}{n(k)} = -\dfrac{\log \lambda}{\log \mu}.$
By applying a similar argument to $\overline f$, one can prove 
$$\displaystyle \lim_{k \to \infty} \bar{\lambda}^{n(k)} \bar{\mu}^{m(k)-(\overline{m}_0-m_0)} 
= \dfrac{h(x_0)}{\bar{a}\bar{z}_0}$$ 
and hence 
$\displaystyle \lim_{k \to \infty} \dfrac{m(k)}{n(k)} = 
\lim_{k \to \infty} \dfrac{m(k)-(\overline{m}_0-m_0)}{n(k)} =-\dfrac{\log \bar{\lambda}}{\log \bar{\mu}}.$
Consequently, 
$\dfrac{\log \lambda}{\log \mu} = \dfrac{\log \bar{\lambda}}{\log \bar{\mu}}$ holds.
\end{proof}

\begin{lemma}\label{l_h_C^1}
If $\dfrac{\log \lambda}{\log \mu}$ is irrational, then the restriction 
$h|_{W_+^u(p)}$ is locally $C^1$-diffeomorphic,
where $W_+^u(p)$ is the component of $W^u(p)\setminus \{p\}$ containing $q$.
\end{lemma}

\begin{proof}
Let $s_n$ be the real number with $\mathrm{pr}_x(r_n)=\mu^{-s_n}$.
Since 
$$\mathrm{pr}_x(r_n) \approx az_0\left(\lambda^{n} + O \bigl(\lambda^{\frac{3}{2}n}\bigr)\right)$$ 
by \eqref{eqn_va} and \eqref{eqn_tn+-},
we have 
$$1=\mathrm{pr}_x(r_n)\mu^{s_n} \approx az_0\left(\lambda^{n}+O\bigl(\lambda^{\frac{3}{2}n}\bigr)\right)\mu^{s_n} \approx az_0\lambda^{n}\mu^{s_n}.$$
Thus $c_n=az_0\lambda^{n}\mu^{s_n}$ satisfies $\lim_{n\to \infty}c_n=1$.
Moreover, 
\begin{equation}\label{eqn_s_n}
s_n=\frac{\log c_n}{\log \mu}-\frac{\log (az_0)}{\log \mu}
-n\frac{\log \lambda}{\log \mu}.
\end{equation}
Since $-\dfrac{\log \lambda}{\log \mu}$ is irrational, 
the set
$$\left\{-\frac{\log (az_0)}{\log \mu}-n\dfrac{\log \lambda}{\log \mu}\mod 1\,;\  n=1,2,\dots \right\}$$
is dense in the interval $[0,1]$.
Since $\lim_{n\to \infty}\log c_n=0$, 
the set $S=\{s_n\mod 1\,;\,n=1,2,\dots\}$ is also dense in $[0,1]$.

Take a point $x_0$ of $[\mu^{-1},1]$ arbitrarily, and let 
$\sigma \in [0,1]$ be the real number with $\mu^{-\sigma}=x_0$.
Since $[\mu^{-1},1]$ is a fundamental domain for $f$ in $W_{\mathrm{loc},+}^u(p)$, 
it follows from the density of $S$ that there exist subsequences $\{n(k)\}$, $\{m(k)\}$ of $\mathbb{N}$ such that 
$\lim_{k\to \infty}(s_{n(k)}-m(k))=\sigma$. 
Then
\begin{equation*}\label{eqn_wsm}
\begin{split}
x_0&=\mu^{-\sigma}=\lim_{k\to \infty}\mu^{-s_{n(k)}+m(k)} =\lim_{k\to \infty} \mathrm{pr}_x(r_{n(k)})\mu^{m(k)}\\
&=\lim_{k\to \infty} az_0 \left( \lambda^{n(k)}+O(\lambda^{\frac{3}{2}n(k)})\right)\mu^{m(k)} =\lim_{k\to \infty} az_0 \lambda^{n(k)}\mu^{m(k)}.
\end{split}
\end{equation*}
Thus we have 
$\displaystyle \lim_{k \to \infty} \lambda^{n(k)} \mu^{m(k)} = \dfrac{x_0}{az_0}.$

Since $\overline{f}$ is conjugate to $f$ via $h$, 
$\mathrm{pr}_x(\overline{r}_{n(k)})\bar{\mu}^{m(k)-(\overline{m}_0-m_0)}$ converges to $h(x_0)$.
As above, we have  
$$\displaystyle \lim_{k \to \infty} \bar{\lambda}^{n(k)} \bar{\mu}^{m(k)-(\overline{m}_0-m_0)} = \dfrac{h(x_0)}{\bar{a}\bar{z}_0}.$$
If we set 
$\tau = \dfrac{\log \bar{\mu}}{\log \mu} =\dfrac{\log \bar{\lambda}}{\log \lambda}$, 
then $\bar{\mu}=\mu^\tau$ and $\bar{\lambda}=\lambda^\tau$. 
It follows that 
$$\dfrac{x^\tau_0}{a^\tau z_0^\tau} = \dfrac{h(x_0)}{\bar{a}\bar{z}_0}\bar\mu^{\overline{m}_0-m_0}.$$
Thus  
$h|_{W^u_{\mathrm{loc},+}(p)}$ is a $C^1$-diffeomorphism 
represented as 
$$h(x) = \dfrac{\bar{a}\bar{z}_0}{a^\tau z^\tau_0\,\bar\mu^{\overline{m}_0-m_0}}x^\tau,$$
where $W^u_{\mathrm{loc},+}(p)$ is the component of $W^u_{\mathrm{loc}}(p)\setminus \{p\}$ containing $q$.
Since $W^u_{+}(p)=\bigcup_{n=0}^\infty f^n(W^u_{\mathrm{loc},+}(p))$ and 
both $f$ and $\overline f$ are $C^3$-diffeomorphisms, 
$h|_{W_+^u(p)}$ is locally $C^1$-diffeomorphic.
This completes the proof.
\end{proof}

\begin{proof}[Proof of (\ref{M2}) of Theorem \ref{thm_A}]
Take a sequence $\{ q_j \}$ on $W^u_{\mathrm{loc},+}(p)$ converging to $q$ and set $t_j = \varphi(q_j)$.
See Figure \ref{fig_x1_2}.
%%%%%%%%%%%%%%%%%%%%%%%%%%%%%%%%%%%%%%%
\begin{figure}[hbt]
\centering
\scalebox{0.3}{\includegraphics[clip]{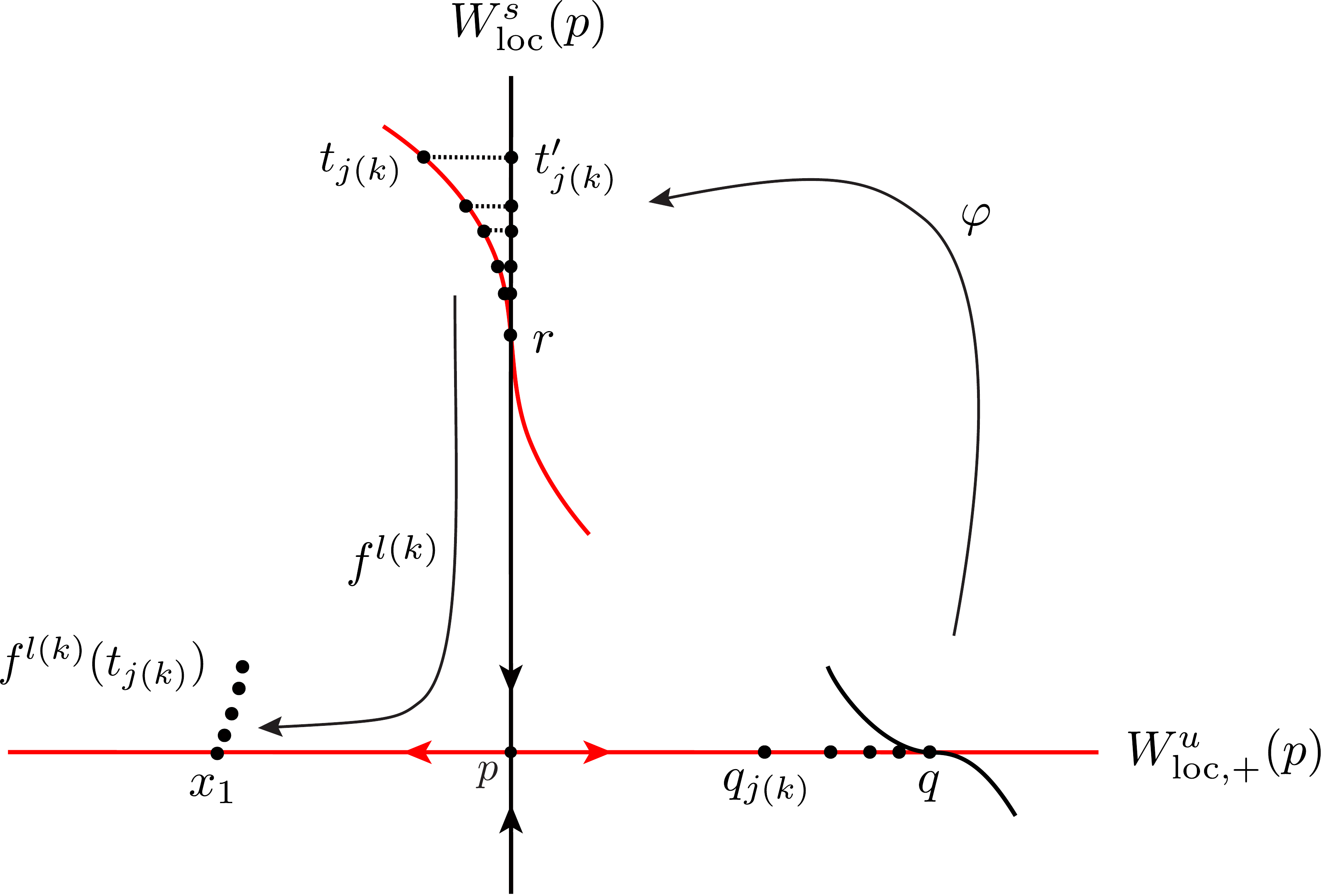}}
\caption{The case of II$_{++}$}
\label{fig_x1_2}
\end{figure}
%%%%%%%%%%%%%%%%%%%%%%%%%%%%%%%%%%%%%%%
Let $t^{\prime}_j$ be the image of $t_j$ by the horizontal projection to $W^s_{{\rm loc}}(p)$.
Obviously, both $t_j$ and $t^{\prime}_j$ converge to $r$ as $k \to \infty$.
There exist subsequences $\{ t_{j(k)} \}$ of $\{ t_j \}$, $\{ l(k) \}$ of $\mathbb{N}$ and a point $x_1 $ of $W^u_{{\rm loc}}(p)$ 
with $\displaystyle \lim_{k \to \infty} f^{l(k)}(t_{j(k)}) = x_1$.
Then the following approximations
$$x_1 \sim {\rm pr}_x(t_{j(k)}) \mu^{l(k)} \sim [ d(t^{\prime}_{j(k)},r)]^3 \mu^{l(k)} 
\sim [ d(q_{j(k)},q)]^3 \mu^{l(k)}$$
hold.
It follows that $\mu^{-l(k)} \sim [ d(q_{j(k)},q)]^3$.
Similarly, we have $\bar{\mu}^{-l(k)+(\overline m_0-m_0)} \sim [ d(\bar{q}_{j(k)},\bar{q})]^3$, 
where $\bar{q}_{j(k)}=h(q_{j(k)})$.
Since $h|_{W^u_{+}(p)}$ is locally $C^1$-diffeomorphic by Lemma \ref{l_h_C^1},  
$$d(\bar{q}_{j(k)},\bar{q}) \sim d(q_{j(k)},q).$$
Thus
$$\biggl( \dfrac{\bar{\mu}}{\mu} \biggr)^{-l(k)} \sim 
\biggl( \dfrac{d(\bar{q}_{j(k)},\bar{q})}{d(q_{j(k)},q)} \biggr)^3\bar\mu^{-(\overline m_0-m_0)} 
\sim 1.$$
This implies that $\mu = \bar{\mu}$. 
By \eqref{M1}, we also have $\lambda = \bar{\lambda}$.
This completes the proof of the part \eqref{M2}.
\end{proof}

\begin{remark}\label{r_one_sided}
Some arguments used in the case that the tangency between $W^s(p)$ and $W^u(p)$ is one-sided (for example 
\cite{pa, dm, po}) can not be applicable to the two-sided case.
Here we explain the reason.

Suppose that a homoclinic tangency $q_0$ is one-sided, say a quadratic tangency.
Take an arc $\gamma$ in $U(q_0)$ meeting $W_{\mathrm{loc}}^u(p_0)$ orthogonally at $q_0$.
Let $\{ w_i \}$ be a sequence in $\gamma$ converging to $q_0$ from above.
Then  
\begin{equation}\label{ineq_dist_1}
d(w_i,W^s(p_0)) \approx d(w_i,W^u_{\mathrm{loc}}(p_0))
\end{equation} 
holds.
On the other hand, their images by the conjugacy homeomorphism $h$ satisfy
\begin{equation}\label{ineq_dist_2}
d(h(w_i),W^s(p_1)) \leq d(h(w_i),W^u_{\mathrm{loc}}(p_1)).
\end{equation}
See Figure \ref{fig_technique}\,(a).
%%%%%%%%%%%%%%%%%%%%%%%%%%%%%%%%%%%%%%%%%%%%%%%%%%%%%%%%%%%%%%%%%%%%%%%%%%%%%%%%%%%%%%%%%%%%%%%%%%%%
\begin{figure}[hbt]
 \centering
 \scalebox{0.5}{\includegraphics[clip]{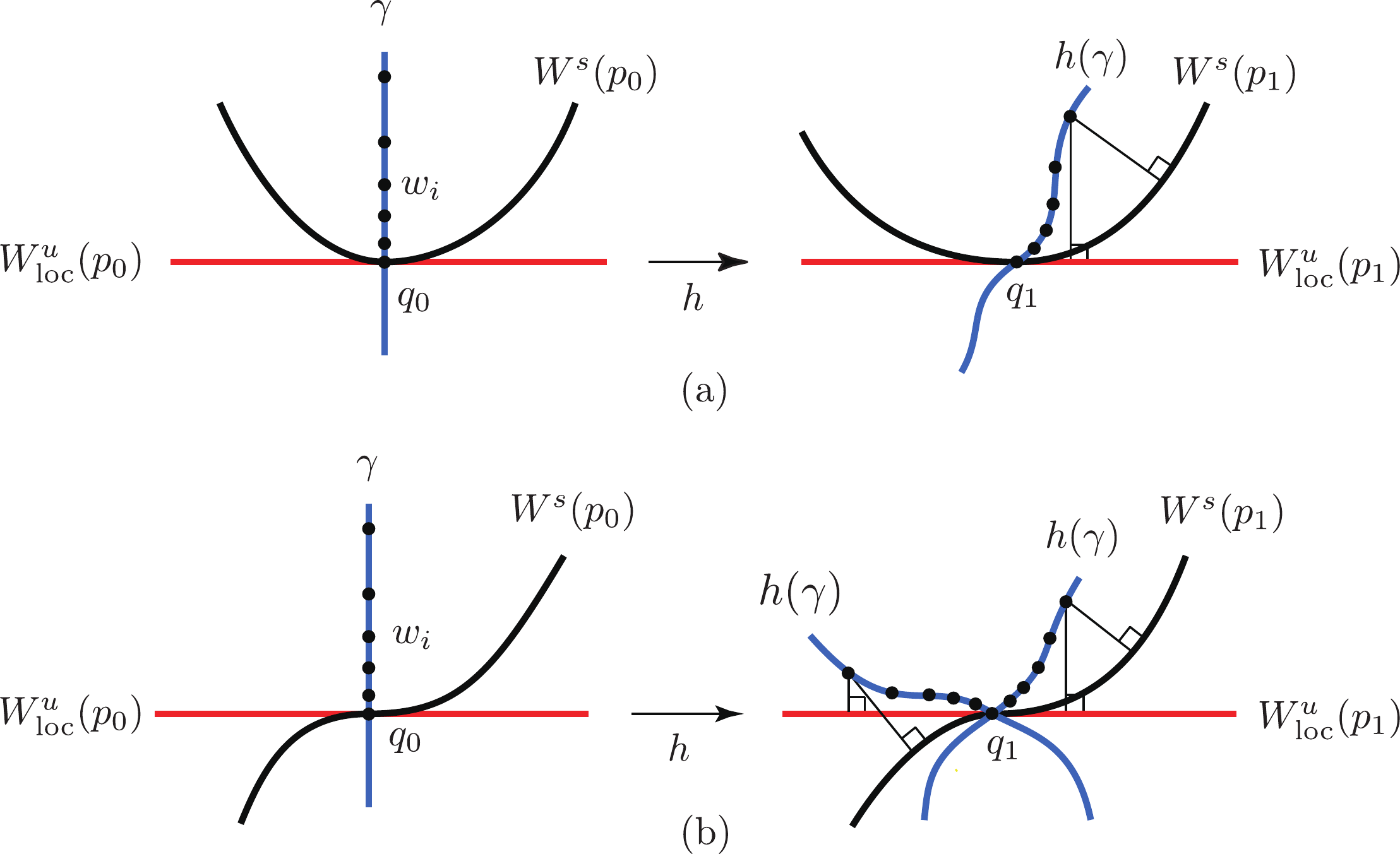}}
 \caption{}
 \label{fig_technique}
\end{figure}
%%%%%%%%%%%%%%%%%%%%%%%%%%%%%%%%%%%%%%%%%%%%%%%%%%%%%%%%%%%%%%%%%%%%%%%%%%%%%%%%%%%%%%%%%%%%%%%%%%%%
By using \eqref{ineq_dist_1} and \eqref{ineq_dist_2}, one can shown that $\dfrac{\log \lambda_1}{\log \mu_1} \leq \dfrac{\log \lambda_0}{\log \mu_0}$.
%, see \cite{pa, dm, po} for details.
By applying the same argument to $h^{-1}$, we also have $\dfrac{\log \lambda_1}{\log \mu_1} \geq \dfrac{\log \lambda_0}{\log \mu_0}$, and hence 
$\dfrac{\log \lambda_1}{\log \mu_1} =\dfrac{\log \lambda_0}{\log \mu_0}$.

Now we consider the case of 2-sided tangencies, say cubic tangencies, and 
$\{ w_i \}$ is a sequence as above.
Then the approximation \eqref{ineq_dist_1} still holds.
However, the inequality \eqref{ineq_dist_2} would not hold as is suggested in Figure \ref{fig_technique}\,(b).
So it might be difficult to get the inequality $\dfrac{\log \lambda_1}{\log \mu_1} \leq \dfrac{\log \lambda_0}{\log \mu_0}$ only by arguments in \cite{pa, dm, po}.
Thus we need another idea in the study of moduli associated with two-sided 
homoclinic tangencies. 
\end{remark}

\section{Adaptable conditions}\label{S_adaptable}
In this section, we will present conditions on the signs of $a$, $bc$, $\lambda$ and $\mu$ 
under which any arguments presented throughout the previous sections are valid.

Recall that we have set
$$U(p)=[-2,2]\times [-2,2],\ W_{\mathrm{loc}}^u(p)=[-2,2]\times \{0\},\ W_{\mathrm{loc}}^s(p)=\{0\}\times [-2,2].$$
The union $W_{\mathrm{loc}}^u(p)\cup W_{\mathrm{loc}}^s(p)$ divides $U(p)$ to four components.
The closures of these components containing $(1,1)$, $(-1,1)$, $(-1,-1)$ and $(1,-1)$ are called 
the first, second, third and fourth quadrants of $U(p)$ and denoted by $Q_1$, $Q_2$, $Q_3$ and $Q_4$, respectively.
In our argument it is required that $\varphi(R_\varepsilon)$ or some substitution is in $Q_1$.
If $R_\varepsilon$ lies in $Q_2$, then we may use 
$$R^-_\varepsilon = [(1+\varepsilon)^{-3}, (1+\varepsilon)^{-1}] \times [0, \varepsilon^3]$$
instead of $R_\varepsilon$.
Then $\varphi(R^-_\varepsilon)$ is in $Q_1$.
See Figure \ref{fig_vpRe}.
%%%%%%%%%%%%%%%%%%%%%%%%%%%%%%%%%%%%%%%%%%%%%%%%%%%%%%%%%%%%%%%%%%%%%%%%%%%%%%%%%%
\begin{figure}[hbt]
\centering
\scalebox{0.3}{\includegraphics[clip]{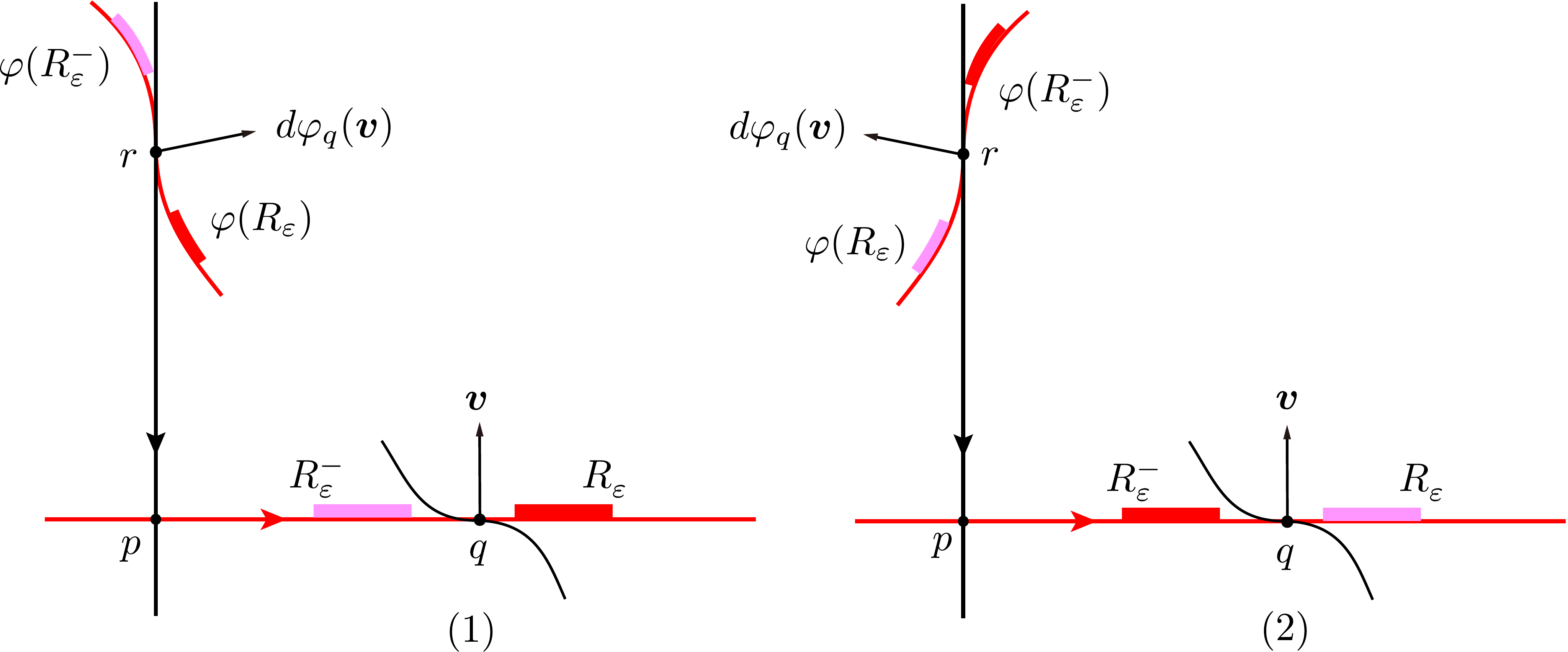}}
\caption{}%(1) $R_\varepsilon$ ((2) $R_\varepsilon^-$) is used in the proof.}
\label{fig_vpRe}
\end{figure}
%%%%%%%%%%%%%%%%%%%%%%%%%%%%%%%%%%%%%%%%%%%%%%%%%%%%%%%%%%%%%%%%%%%%%%%%%%%%%%%%%%
Thus one can arrange the placement of $\varphi(R_\varepsilon)$ suitably under any conditions on the signs of $a$, $bc$, $\lambda$ and $\mu$.

\begin{definition}[Adaptable condition]\label{def_adaptable}
$f$ satisfies the \emph{adaptable condition} with respect to $(p,q)$ if, 
for all sufficiently large positive integers $n$ (or positive even or odd integers), 
there exists a rectangle $S_n$ defined as in Section \ref{S_Sbox} and 
either $S_n$ or its image $f(S_n)$ lies in $Q_1$. 
\end{definition}

As was seen in Section \ref{S_Sbox}, $S_n$ exists if and only if 
there exists $t_n$ satisfying the condition
\begin{equation}\tag{\ref{eqn_3ctn}'}\label{eqn_3ctn'}
3ct_n^2\approx -b\lambda^nz_0
\end{equation}
which corresponding to \eqref{eqn_3ctn}.
Here $z_0$ is the positive constant as illustrated in Figure \ref{fig_alpha_n}.

Now we will see that the existence of $S_n$ and the placements of $S_n$ and $f(S_n)$ are strictly determined by the signs of $a$, $bc$, $\lambda$ and $\mu$, 
which are classified to the sixteen cases as in Table \ref{table_1}
%%%%%%%%%%%%%%%%%%%%%%%%%%%%%%%%%%%%%%%%%%%%%%%%%%%%%%%%%%%%%%%%%%%%%%%%%%%%%%%%%%%%%
\begin{table}[hbt]
\centering  
{\renewcommand{\arraystretch}{1.2}
\begin{tabular}{|c|c|c||c|c|c|c|}
\hline
\multicolumn{3}{|c||}{Case}& $a$ & $bc$ &$\lambda$&$\mu$\\
\hline
& I$_+$&  I$_{++}$ &  &  & $+$& $+$\\
\cline{3-3}\cline{7-7}
\quad I\quad{} & &  I$_{+-}$ & $+$ & $+$ & & $-$\\
\cline{2-3}\cline{6-7}
& I$_-$&  I$_{-+}$ &  &  & $-$& $+$\\
\cline{3-3}\cline{7-7}
& &  I$_{--}$ &  &  & & $-$\\
\hline
& II$_+$&  II$_{++}$ &  &  & $+$& $+$\\
\cline{3-3}\cline{7-7}
II& &  II$_{+-}$ & $+$ & $-$ & & $-$\\
\cline{2-3}\cline{6-7}
& II$_-$&  II$_{-+}$ &  &  & $-$& $+$\\
\cline{3-3}\cline{7-7}
& &  II$_{--}$ &  &  & & $-$\\
\hline
& III$_+$&  III$_{++}$ &  &  & $+$& $+$\\
\cline{3-3}\cline{7-7}
III& &  III$_{+-}$ & $-$ & $+$ & & $-$\\
\cline{2-3}\cline{6-7}
& III$_-$&  III$_{-+}$ &  &  & $-$& $+$\\
\cline{3-3}\cline{7-7}
& &  III$_{--}$ &  &  & & $-$\\
\hline
& IV$_+$&  IV$_{++}$ &  &  & $+$& $+$\\
\cline{3-3}\cline{7-7}
IV& &  IV$_{+-}$ & $-$ & $-$ & & $-$\\
\cline{2-3}\cline{6-7}
& IV$_-$&  IV$_{-+}$ &  &  & $-$& $+$\\
\cline{3-3}\cline{7-7}
& &  IV$_{--}$ &  &  & & $-$\\
\hline
\end{tabular}
}
\bigskip 
\caption{}
\label{table_1}
\end{table}
%%%%%%%%%%%%%%%%%%%%%%%%%%%%%%%%%%%%%%%

First we suppose that $\lambda>0$.
Then there exists $t_n$ satisfying \eqref{eqn_3ctn'} if and only if 
$bc<0$.
Moreover, if $a>0$, then $S_n$ is in $Q_1$, which belongs to Case II$_+$. 
See Figure \ref{fig_Cases}\,(1).
If $a<0$, then $S_n$ is in $Q_2$.
Hence $f(S_n)$ is $Q_1$ if $\mu<0$, which is in Case IV$_{+-}$.
See Figure \ref{fig_Cases}\,(2).
%%%%%%%%%%%%%%%%%%%%%%%%%%%%%%%%%%%%%%%%%%%%%%%%%%%%%%%%%%%%%%%%%%%%%%%%%%%%%%%%%%
\begin{figure}[hbt]
\centering
\scalebox{0.5}{\includegraphics[clip]{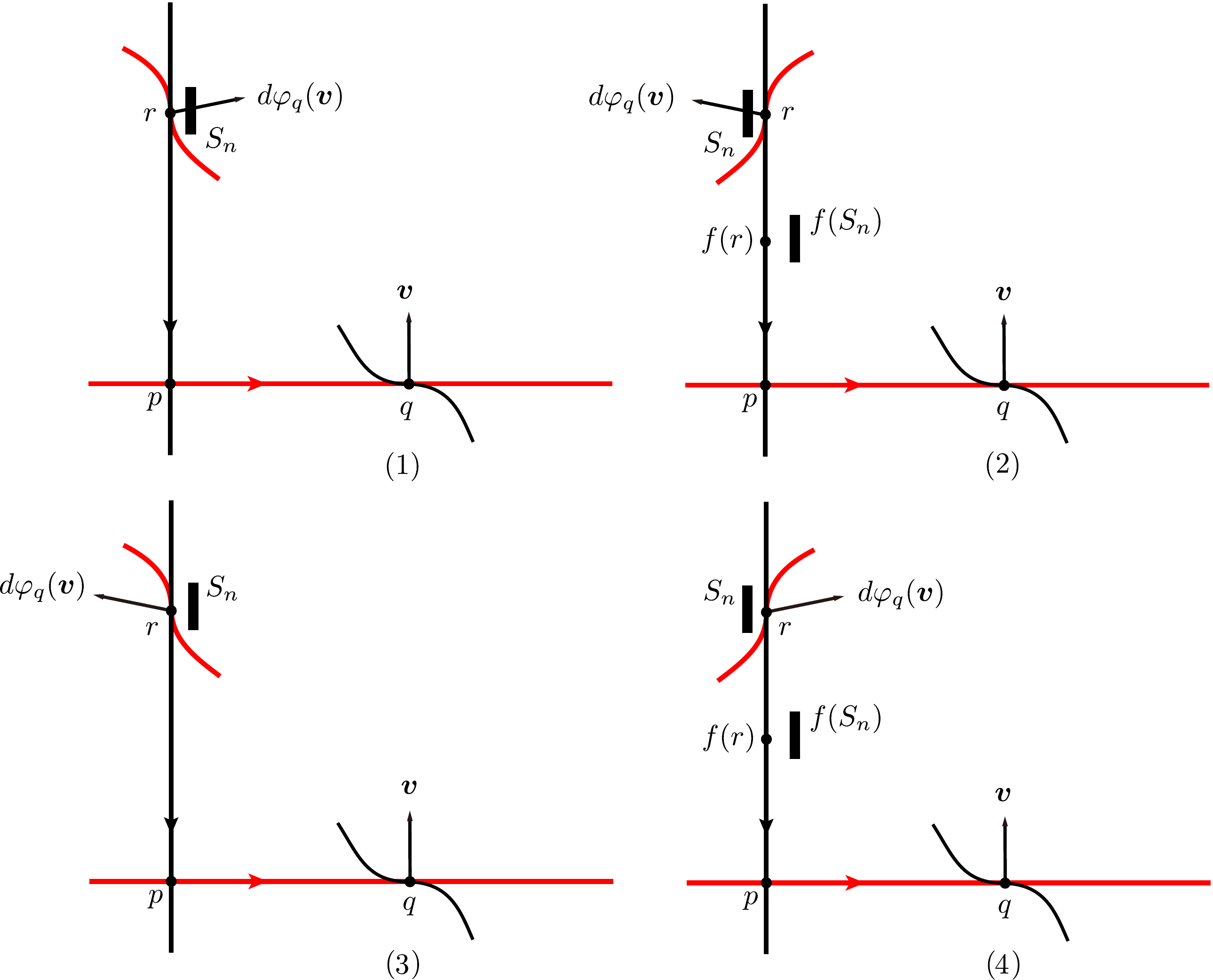}}
\caption{}
\label{fig_Cases}
\end{figure}
%%%%%%%%%%%%%%%%%%%%%%%%%%%%%%%%%%%%%%%%%%%%%%%%%%%%%%%%%%%%%%%%%%%%%%%%%%%%%%%%%%

Next we suppose that $\lambda<0$.
Then there exists $t_n$ satisfying \eqref{eqn_3ctn'} if and only if either 
(i) $bc<0$ and $n$ is even or (ii) $bc>0$ and $n$ is odd.
In the case (i), $S_n$ is in $Q_1$ if $a>0$, which belongs to Case II$_-$.
See Figure \ref{fig_Cases}\,(1).
If $a<0$ and $\mu<0$, then $f(S_n)$ is in $Q_1$, which belongs to Case IV$_{--}$.
See Figure \ref{fig_Cases}\,(2).
On the other hand, in the case (ii), $S_n$ is in $Q_1$ if $a<0$, which belongs to Case III$_-$.
See Figure \ref{fig_Cases}\,(3).
If $a>0$ and $\mu<0$, then $f(S_n)$ is in $Q_1$, which belongs to Case I$_{--}$.
See Figure \ref{fig_Cases}\,(4).

Thus we have the following preposition.

\begin{proposition}\label{p_adaptable}
If one of Cases {\rm I$_{--}$, II, III$_{-}$, IV$_{--}$} and {\rm IV$_{+-}$} holds, 
then $f$ satisfies the adaptable condition with respect to $(p,q)$. 
\end{proposition}

It follows from the proposition that $f$ satisfies the adaptable condition in nine of the sixteen cases 
in Table \ref{table_1}.

%%%%%%%%%%%%%%%%%%%%%%%%%%%%%%%%%%%%%%%%%%%%%%%%%%%%%%%%%%%%%%%%%%%%%%%%%%%%%%%%%%%%%%%%%%

\end{document}